\documentclass[reqno]{amsart}
\usepackage{mathrsfs}
\usepackage{amsmath}
\usepackage{amsthm}
\usepackage{amsfonts}
\usepackage{amssymb}
\usepackage{url}
\usepackage{enumerate}
\usepackage[pdftex,bookmarks=true]{hyperref}
  \usepackage[usenames, dvipsnames]{color}
\usepackage{verbatim}
\usepackage{graphicx}

\usepackage{pdfsync}

\renewcommand\Re{{\operatorname{Re}}}

\newcommand\Vol{{\operatorname{Vol}}}

\newcommand\R{{\mathbb{R}}}

\renewcommand\P{{\mathbb{P}}}
\newcommand\E{{\mathbb{E}}}

\newcommand\Z{{\mathbb{Z}}}

\newcommand\F{{\mathbb{F}}}

\newcommand\al{\alpha}

\newcommand\1{\mathbf{1}}

\newcommand\Ba{{\mathbf a}}
\newcommand\Bb{{\mathbf b}}

\newcommand\Be{{\mathbf e}}

\newcommand\Bn{{\mathbf n}}

\newcommand\Bs{{\mathbf s}}
\newcommand\Bt{{\mathbf t}}
\newcommand\Bu{{\mathbf u}}
\newcommand\Bv{{\mathbf v}}
\newcommand\Bw{{\mathbf w}}
\newcommand\Bx{{\mathbf x}}
\newcommand\By{{\mathbf y}}

\newcommand\CB{{\mathcal B}}
\newcommand\CE{{\mathcal E}}
\newcommand\CF{{\mathcal F}}
\newcommand\CG{{\mathcal G}}

\newcommand\CN{{\mathcal N}}

\newcommand\CP{{\mathcal P}}

\newcommand\CU{{\mathcal U}}

\newcommand\eps{\varepsilon}

\newcommand\lang{\langle}
\newcommand\rang{\rangle}

\newcommand\bs{\backslash}

\newcommand{\mfn}{\mathfrak{n}}
\newcommand{\bmu}{\boldsymbol\mu}
\newcommand{\bsig}{\boldsymbol\Sigma}

\usepackage{fullpage}

\theoremstyle{plain}
  \newtheorem{theorem}[subsection]{Theorem}

    \newtheorem{proposition}[subsection]{Proposition}
  
  \newtheorem{lemma}[subsection]{Lemma}
  \newtheorem{corollary}[subsection]{Corollary}

  \newtheorem{remark}[subsection]{Remark}
  
  \newtheorem{claim}[subsection]{Claim}

\theoremstyle{definition}
  \newtheorem{definition}[subsection]{Definition}

\begin{document}

\title{A note on the singularity probability of random directed $d$-regular graphs}

 \author{Hoi H. Nguyen}
 \address{Department of Mathematics\\ The Ohio State University \\ 231 W 18th Ave \\ Columbus, OH 43210 USA}
\email{nguyen.1261@math.osu.edu}
\thanks{The authors are supported by the NSF CAREER grant DMS-1752345.}

 \author{Amanda Pan}
\address{Department of Mathematics\\ The Ohio State University \\ 231 W 18th Ave \\ Columbus, OH 43210 USA}
\email{pan.754@osu.edu}
\maketitle


\begin{abstract} In this note we show that the singular probability of the adjacency matrix of a random $d$-regular graph on $n$ vertices, where $d$ is fixed and $n \to \infty$, is bounded by $n^{-1/3+o(1)}$. This improves a recent bound by Huang in \cite{H2}. Our method is based on the study of the singularity problem modulo a prime developed in \cite{H2} (and also partially in \cite{M18, NgW-d}), together with an inverse-type result on the decay of the characteristic function. The latter is related to the inverse Kneser's problem in combinatorics.  
                \end{abstract}

\section{Introduction}

The singularity problem in combinatorial random matrix theory states that if a square matrix $A_n$ of size $n$ is ``sufficiently
random'', then $A_n$ is non-singular asymptotically almost surely as
$n$ tends to infinity, in other words $p_n$, the probability of $A_n$
being singular, tends to zero. This problem has a rich history, for which
we now mention briefly. In the early 60s Koml\'os \cite{Komlos67}
showed that if the entries of $A_n$ take values $\{0,1\}$
independently with probability 1/2 then $p_n=O(n^{-1/2})$. This bound
was significantly improved to exponential bounds of type $(1-\eps)^n$ by Kahn, Koml\'os and Szemer\'edi
\cite{KKSz95} in 1995, by Tao
and Vu \cite{TV07} in 2007, by Rudelson and
Vershynin \cite{RV08} in 2008, and by Bourgain, Vu and Wood \cite{BVW10} in 2010. More recently, Tikhomirov \cite{Tikh20} has obtained a nearly optimal bound $p_n= (\frac{1}{2}+o(1))^n$. The methods of these results also give exponential
bounds for other more general iid ensembles. Since then, there have been
subsequent papers addressing the sparse cases, such as \cite{PWood12}, \cite{BR17}, \cite{HH}, \cite{CEG}, \cite{LT}, \cite{JSS}. We refer the reader to 
these papers and the references therein to various extension and application of the
singularity problem for the iid models.

In another direction, there
have been results studying the singularity problem for matrices with various dependency conditions on the entries. For
instance in \cite{Nguyen13} the first author studied random doubly stochastic matrices, or in  \cite{AChW16} Adamczak, Chafai and Wolff studied
random matrices with exchangeable entries. More relatedly,
Cook \cite{Cook2017} studied the singularity of $A_{n,d}$, the adjacency matrix of a random directed $d$-regular graph, where he showed that $p_n =d^{-\Omega(1)}$
as long as $\min(d,n-d) \geq C \log^2 n$ for some absolute constant
$C$. A similar result was also established by Basak, Cook and
Zeitouni \cite{BCZ18} for sum of $d$ random permutation
matrices as long as $d \ge \log^{12-o(1)} n$. While these results are
highly non-trivial, the random matrices are still relatively
dense. For smaller $d$, the recent work by Litvak, Lytova, Tikhomirov,
Tomczak-Jaegermann and Youssef in \cite{Litvak2017} shows that $p_n \leq \frac{C\log^3 d}{\sqrt{d}}$ as long as $C\leq d \leq
cn/\ln^2 n$ for some constants $c,C$. As a consequence, this bound implies that $p_n \to
\infty$ if $d \to \infty$. By a
more involved study of the structure of the eigenvectors of matrices of $A_{n,d}$, it has been
shown by the same group of authors in \cite{Litvak2018} that
asymptotically almost surely the rank of $A_{n,d}$ is at least $n-1$
as long as  $d>C$ for sufficiently large constant $C$. Finally, very
recently Huang \cite{H2}, M\'esz\'aros \cite{M18}  (see also \cite{NgW-d}) confirmed the conjecture by Vu \cite{Vu2014} that $p_n \to \infty$ as $n\to \infty$ for the $A_{n,d}$ model with fixed $d$.\footnote{We also refer the reader to \cite{CEG, FKSS} for results regarding other models of extremely sparse graphs.} The following quantitative result was shown in \cite[Theorem 1.3]{H2}.

\begin{theorem}\label{thm1} Let $d\ge 3$ be a fixed integer. Then if $n$ sufficiently large, for a random $d$-regular directed graph on $n$ vertices, the probability $p_n$ that its adjacency matrix $A_{n,d}$ is singular is
$$p_n \leq n^{-\min\{1/4, (d-2)/(2d)\}}.$$
\end{theorem}
In particular, when $d=3$ the above gives $O(n^{-1/6})$.

The papers \cite{H2, M18, NgW-d} also addressed the symmetric case, which is more complicated and is not the main focus of our current paper. As the reader can see, although there have been massive contributions on the quantitative aspect of the singularity bound for various (not very sparse) random matrix models, the above paper \cite{H2} is the only reference that produces a quantitative estimate for $p_n$ of $A_{n,d}$. In the current note we further explore this quantitative direction by showing
\begin{theorem}[Main result]\label{thm:main:directed} Let $\eps>0$ be given. Let $d\ge 3$ be fixed. Then for sufficiently large $n$, for a random $d$-regular directed graph on $n$ vertices, the probability that $A_{n,d}$ is singular is bounded by 
$$p_n \le n^{-1/3+\eps}.$$
\end{theorem} 
Hence with respect to the model $A_{n,d}$, our result improves over the $n^{-1/4}$ barrier from Theorem \ref{thm1} for all $d\ge 2$. With a more careful  analysis, we can also replace the bound $n^{-1/3+\eps}$ by $Cn^{-1/3}$ for some sufficiently large constant $C$, but our bounds are still far from being best possible, where it seems the bound for $p_n$ should be of order $1/n^{d-2}$, which would mean that the singularity event is mainly from the cases of having two identical rows or two identical columns (see Figures [1] and [2]). It is desirable to establish similar probability bound for the least singular value of $A_{n,d}$, for which the current approach does not seem to work. 

Our approach mainly follows the method of \cite{H2} which studies the singularity of the matrix $A_{n,d}$ over $\Z/p\Z$ for some large $p$. In this approach we will consider $\P(A_{n,d}\Bv =0)$ for each fixed non-zero $\Bv \in (\Z/p\Z)^n$. We hope that the probability is still small after taking union bound over all non-zero choices of $\Bv$ (modulo its direction). A somewhat similar strategy was also carried out in \cite{M18, NgW-d} for the cokernel statistics of $A_{n,d}$ as an integral matrix. Our new contribution shows an interesting relation between the decay of the characteristic functions of a special family of random walks arises from the configuration model of $A_{n,d}$ and an inverse-type Kneser problem in combinatorics (Theorem \ref{thm:inverse:d}). More specifically, we extend the treatment of \cite{H2} on the central limit theorem (Proposition \ref{prop:clt:directed}) and on the tail bound estimate (Proposition \ref{prop:dev:directed}) to a broader range $p \le n^{1/3-o(1)}$. Least but not last, it is an interesting problem to extend the treatments to larger $p$, a problem which is directly related to the upper bound of $p_n$, but is also useful toward the study of $\Z$-statistics of the cokernels of $A_{n,d}$.

\begin{figure}[h!] \label{fig:1}
	\includegraphics[width=.7\textwidth]{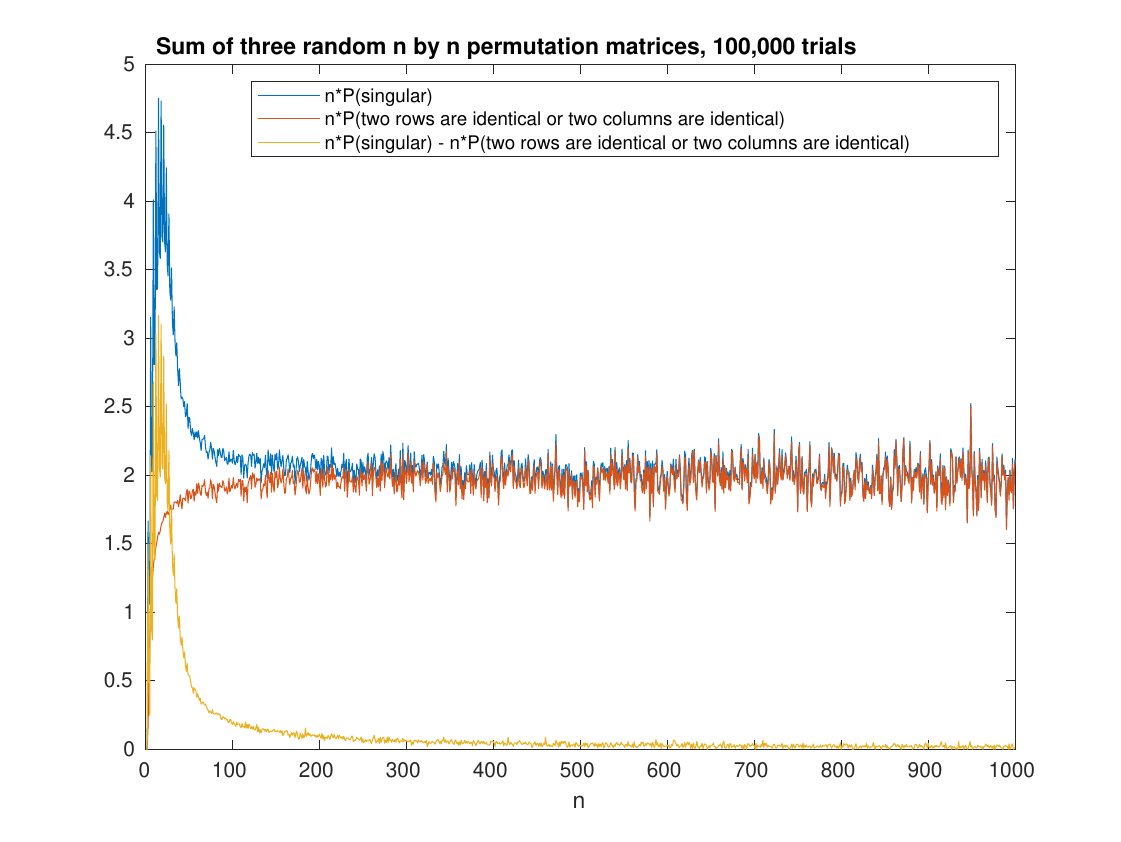} 
	\caption{Sum of three random permutation matrices}
\end{figure}

\begin{figure}[h!] \label{fig:2}
	\includegraphics[width=.7\textwidth]{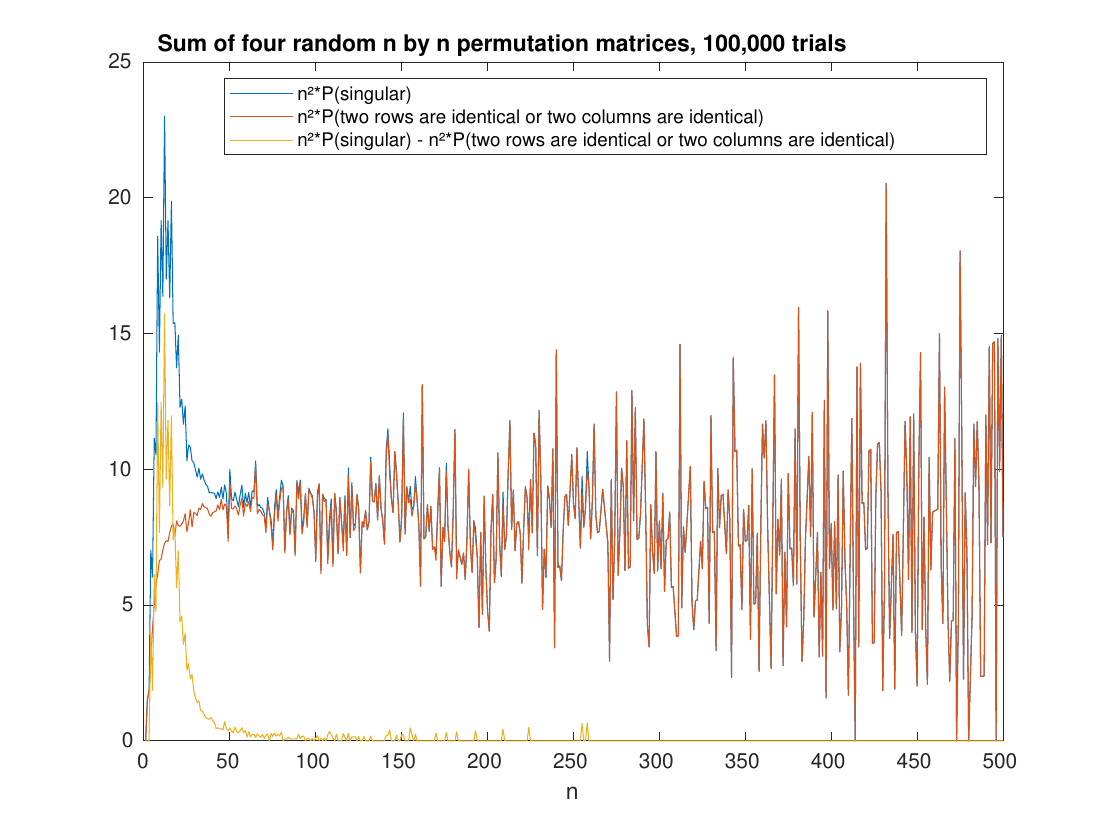}
	\caption{Sum of four random permutation matrices}
\end{figure}


{\bf Notations.} We say that $X \asymp Y$ if $X=O(Y)$ and $Y=O(X)$. We say that $X=\Omega(Y)$ if $X \ge CY$ for some absolute positive constant $C$. Given a parameter $\al$, we say that $X = O_\al (Y)$, or $X \ll_\alpha Y$, if $X \le CY$ and $C$ is allowed to depend on $\alpha$. 

For any $x\in \R$, we define $\|x\|:=\|x\|_{\R/\Z}$ to be the distance of $x$ to the nearest integer. We define $e(x) = \exp(2\pi \sqrt{-1}x )$ and $e_p(x) = \exp(2\pi \sqrt{-1}x /p)$. 

Finally, if not specified otherwise, the parameter $n$ in this note is assumed to be sufficiently large.

\section{Some formulas and the proof method }


As mentioned, the singularity problem views the $M$ as matrices over $\R$, but if the entries are integers they could also be viewed as elements of the field $\Z/p\Z$ for any prime $p$.  A matrix is singular mod $p$ exactly when its determinant is $0$ mod $p$, and so heuristically, one expects this to happen about $1/p$ of the time instead of $0\%$ of the time. This was the motivation for the treatments of \cite{H2,M18, NgW-d}. In what follows we closely follow the approach of \cite{H2}.

We first use work of Bollob\'as \cite{Bollobas1980}, to replace $A_{n,d}$ with a random multi-graph $A^*_{n,d}$ given as follows (see \cite[Corollary 2.18]{Bollobas2001}). We associate to each vertex $k \in \{1,\dots, n\}$ a fiber $F_k$ of $d$ points and select a permutation $\CP$ of the $nd$ points uniformly at random. Then for each vertex $k \in \{1,\dots, n\}$ and point $k' \in F_k$ we add a directed edge from $k$ to vertex $\ell$ if the points $\CP(k')$ belongs to the fiber $\CF_{\ell}$. By  \cite{Bollobas1980}, for any fixed $d$ the probability that 
$A^*_{n,d}$ has a loop or multiple edge is bounded away from $1$. Hence it suffices to prove the theorem with $A^*_{n,d}$ replaced by $A_{n,d}$. Without loss of generality, in what follows by the configuration model $A_{n,d}$ we mean the model $A^*_{n,d}$.

For a vector $\Bx=(x_1,\dots, x_d)\in \F_p^d$ with $n_j$ components $x_i$ of value $j$, we define 
$$\Phi(\Bx):=(n_0,\dots, n_{p-1}).$$ 

Thus we have 
$$\sum_j n_j =d \mbox{ and } \sum_j j n_j = x_1+\dots+x_d.$$
Given $n_0,\dots, n_{p-1}$  where $\sum_i n_i =n$ we denote by $S_{n_0,\dots, n_{p-1}}$ the set of vectors $\Bv=(v_1,\dots, v_n)$ where for each $i=0,\dots, p-1$ there are exactly $n_i$ entries $i$ in $(v_1,\dots, v_n)$; so there are $\binom{n}{n_0,\dots, n_{p-1}}$ such vectors. 

Let $\CU_{d,p}$ be the multi-set
$$\CU_{d,p}:=\Big\{\Phi(\Bx): \Bx\in \F_p^d, \sum_{i=1}^d x_i =0\Big \}.$$
Hence $| \CU_{d,p}| = p^{d-1}$. For instance, when $d=3$ the vectors $(3,0,\dots,0), (1,1,0,\dots,0,1),$ and $ (1,0,1,0,\dots, 0,1,0)$ all belong to $\CU_{3,p}$.

We have the following beautiful random walk interpretation (see \cite[Proposition 2.1]{H2}).

\begin{claim}\label{claim:1} Given $n_0,\dots, n_{p-1}$, and given  $\Bv \in S_{n_0,\dots, n_{p-1}}$, for the configuration model $A_{n,d}$ on random $d$-regular directed graphs we have
$$\Big|\{M \in A_{n,d}: M \Bv =0\}\Big | = \prod_{j=0}^{p-1} (d n_j)!  \Big |\{(\Bu_1,\dots, \Bu_n) \in  \CU_{d,p}^n: \Bu_1+\dots+\Bu_n=d(n_0,\dots,n_{p-1})\}\Big|$$
$$= \prod_{j=0}^{p-1} (d n_j)! p^{(d-1)n} \P(X_1+\dots+ X_n = (dn_0,\dots, dn_{p-1})),$$
where $X_1,\dots, X_n$ are independent copies of $X$, which is uniformly distributed over $\CU_{d,p}$.
\end{claim}


\subsection{Proof methods} As shown by the above interpretation, it boils down to understanding the random variable $X$. It is elementary to show
$$\bmu = \E X = (d/p,\dots, d/p)$$
and
$$\bsig = \E ((X -\bmu)(X-\bmu)^t) = \frac{d}{p} I - \frac{d}{p^2} \1 \1^t.$$
Also, the characteristic function of $X$ and $X-\mu$ are defined as
$$\phi_X(\Bs) = \frac{1}{p^{d-1}} \sum_{\Bw \in \CU_{d,p}} e^{i \Bs \cdot \Bw}$$
and 
$$\phi_{X-\bmu}(\Bs) = \E \exp(i \Bt \cdot (X-\bmu)) =\exp(- i \Bt \cdot \bmu) \phi_X(\Bs), \Bs\in \R^p.$$
For instance when $d=3$, we have 
$$\phi_X(\Bs) = \frac{1}{p^2} \sum_{a,b \in \Z/p\Z} e^{i (s_a+s_b+s_{-a-b})}.$$

Because of Claim \ref{claim:1} and because $|A_{n,d}| = (nd)!$, in order to prove the singularity probability to be small we aim to show that 
\begin{equation}\label{eqn:totalsum}
\sum_{\substack{(n_0,\dots, n_{p-1}) \in \Z_{\ge 0}, n_0<n \\ \sum_i n_i = n}} \binom{n}{n_0,\dots, n_{p-1}}\frac{p^{(d-1)n} \prod_{j=0}^{p-1} (d n_j)!}{(dn)!} \P(X_1+\dots+ X_n= (dn_0,\dots, dn_{p-1})) \mbox{ is small}.
\end{equation}

\begin{definition} Let $b>0$ be chosen to be sufficiently large, and let $\CE=\CE_b$ be the set of vectors satisfying
\begin{equation}\label{eqn:equi}
\sum_{j=0}^{p-1} (\frac{n_j}{n} - \frac{1}{p})^2 \le \frac{b  \log n}{n}.
\end{equation}
We will call such vectors {\it equidistributed}. 
\end{definition}

Let $\CN$ be the set of $p$-tuples $(n_0,\dots, n_{p-1})$ which are not  $(n,0,\dots, 0)$ and {\em not} equidistributed. Our main result can be deduced from the following two key propositions.
\begin{proposition}[Deviation estimate for the error term]\label{prop:dev:directed} The contribution in \eqref{eqn:totalsum} from $\CN$ is bounded by $o(1)$ as long as $p \le n^{1/3-\eps}$.
\end{proposition}
Note that this result improves upon \cite[Proposition 3.2]{H2} where a similar statement was proved that $p \le n^{(d-2)/2d}$. As a consequence, to justify Theorem \ref{thm:main:directed} it suffices to work with equidistributed vectors. For this we show
\begin{proposition}[Local limit theorem for the main term]\label{prop:clt:directed} The contribution in \eqref{eqn:totalsum} from equidistributed vectors is at most $1+o(1)$ as long as $p\le n^{1/3-\eps}$.
\end{proposition}

We note that with some extra work it might be possible to actually prove that the contribution is $1+o(1)$, see Remark \ref{remark:=}. The above result slightly improves \cite[Proposition 3.1]{H2} where the author there worked with $p \le n^{1/4}$.

We will prove Proposition \ref{prop:dev:directed} in Section \ref{section:dev} and Proposition \ref{prop:clt:directed} in Section \ref{section:clt}, in what follows we deduce our main result.

\begin{proof}(of Theorem \ref{thm:main:directed}) Note that  if $M \in A_{n,d}$ is singular then there exists a non-zero vector $\Bv$ so that $M\Bv=0$ (and hence $M(t\Bv)=0$ for $t=1,\dots, p-1$). We thus have
\begin{align*}
& (p-1) \P(M \in A_{n,d} \mbox{ is singular})  \le \frac{1}{(nd)!}\sum_{M\in A_{n,d}} \sum_{\Bv \neq 0} \1_{M\Bv=0} =  \sum_{\Bv \neq 0} \Big|\{M \in A_{n,d}: M \Bv =0\}\Big | \\
&=  \sum_{\substack{(n_0,\dots, n_{p-1}) \in \Z_{\ge 0}, n_0<n \\ \sum_i n_i = n}} \binom{n}{n_0,\dots, n_{p-1}}\frac{p^{(d-1)n} \prod_{j=0}^{p-1} (d n_j)!}{(dn)!} \P(X_1+\dots+ X_n= (dn_0,\dots, dn_{p-1})) \\
&=  \sum_{(n_0,\dots, n_{p-1})  \notin \CE} ... +  \sum_{(n_0,\dots, n_{p-1}) \in \CE} ... \\
& \le  o(1) + 1+o(1) = 1+o(1).
\end{align*}
Hence
$$\P(M \in A_{n,d} \mbox{ is singular}) \le \frac{1+o(1)}{p-1} =O( n^{-1/3+\eps}).$$
\end{proof}

{\bf Choices of $p,\delta$.} Here and later, $\eps$ is a sufficiently small constant. If not specified otherwise we will assume 
\begin{equation}\label{eqn:p1}
\delta = p^{-(1+3\eps)} \mbox{ and } p^{3(1+2\eps)} \asymp n.
\end{equation}

\section{Treatment over equidistributed vectors: proof of Proposition \ref{prop:clt:directed}}\label{section:clt} 

There are two parts in the sums to analyze, we will give some preliminary discussion on these two parts separately: (i) Stirling formulas for the term $ \binom{n}{n_0,\dots, n_{p-1}}\frac{p^{(d-1)n} \prod_{j=0}^{p-1} (d n_j)!}{(dn)!} $ and (ii) the behavior of \\$(1/p)^{d-1} \sum_{k_1,\dots, k_d, \sum k_i=0 } e_p(s_{k_1}+\dots+s_{k_d})$ for different choices of $(s_0,\dots, s_{p-1})$.

\subsection{Stirling formulas} We first recall the following Stirling bound by Robbins for all positive integer $l$,
\begin{equation}\label{Stirling}
\sqrt{2\pi l} (l/e)^l e^{\frac{1}{12l+1}} < l! < \sqrt{2\pi l} (l/e)^l e^{\frac{1}{12l}}.
\end{equation}
So
 \begin{align*}
\binom{n}{n_0,\dots, n_{p-1}}\frac{p^{(d-1)n} \prod_{j=0}^{p-1} (d n_j)!}{(dn)!} &=p^{(d-1)n}\frac{n!}{\prod_j n_j!} \frac{\prod_{j=0}^{p-1} (d n_j)!}{(dn)!}(1+o(1))\\
&=p^{(d-1)n}\frac{ \sqrt{2\pi n} (n/e)^n e^{\frac{1}{12n}}}{ \prod  \sqrt{2\pi n_j} (n_j/e)^{n_j} e^{\frac{1}{12n_j}}} \frac{ \prod  \sqrt{2\pi dn_j} (dn_j/e)^{dn_j} e^{\frac{1}{12dn_j}} }{  \sqrt{2\pi dn} (dn/e)^{dn} e^{\frac{1}{12dn}}}(1+o(1))\\
&= e^{\frac{1}{12n}- \frac{1}{12dn} + \sum_j \frac{1}{12dn_j} - \frac{1}{12n_j} } \times (\sqrt{d})^{m-1} \times \frac{\prod n_j^{(d-1)n_j }}{(\frac{n}{p})^{(d-1)n}}(1+o(1))\\
&= e^{\frac{1}{12n}- \frac{1}{12dn} + \sum_j \frac{1}{12dn_j} - \frac{1}{12n_j} } \times (\sqrt{d})^{p-1} \times [\prod_{j=0}^{p-1} (\frac{n_j}{n/p})^{n_j}]^{d-1}(1+o(1)).
\end{align*} 

We will work with $\prod_{j=0}^{p-1} (\frac{n_j}{n})^{n_j}$. Write $\mfn_j = \frac{n_j}{n}$.
Recall that 
$$\sum_{j=0}^{p-1} (\frac{n_j}{n} - \frac{1}{p})^2 \le \frac{b  \log n}{n}.$$
Hence trivially 
$$|\frac{n_j}{n} - \frac{1}{p}| = O(1/\sqrt{n}), |\frac{pn_j}{n} - 1| =O(p/\sqrt{n}) =o(1).$$
Hence 
$$\sum_j 1/n_j =O(p^2/n) =o(1).$$
Using this, we obtain (see also \cite[Eqn (3.4)]{H2})
$$ \binom{n}{n_0,\dots, n_{p-1}}\frac{p^{(d-1)n} \prod_{j=0}^{p-1} (d n_j)!}{(dn)!} =\Big(1+o(1) \Big) d^{(p-1)/2} \exp\Big((d-1)n\sum_j \mfn_j \log \mfn_j + \log p)\Big).$$
Note that Taylor expansion for $|h|<1$ shows
$$(h+1) \log (h+1) = h+h^2/2 - h^3(1/2-1/3) + h^4(1/3-1/4) +\cdots.$$
Hence, because $|(pn_j/n)-1|=o(1)$
$$n_j \log ((n_j/n)/(1/p)) = n_j \log [(pn_j/n-1) + 1] =  (n/p) \times (pn_j/n -1 +1) \log [(pn_j/n-1) + 1] $$
$$= (n/p) [(pn_j/n-1) + (pn_j/n-1)^2/2+ \sum_{k=3}^\infty \frac{(-1)^k}{(k-1)k} (pn_j/n-1)^k].$$
So
\begin{align*}
& \sum_j n_j \log ((n_j/n)/(1/p))  =(n/p) [\sum_j (pn_j/n-1)^2/2+ \sum_{k=3}^\infty \frac{(-1)^k}{(k-1)k} (pn_j/n-1)^k].
\end{align*}
We will use the above expansion for $\sum_j \mfn_j \log \mfn_j $. One can see that for equidistributed vectors the terms $\sum_j (pn_j/n-1)^2$ and  $ n (d-1) \sum_j (pn_j/n-1)^3$ are the main contributions, while the contributions from higher order terms are bounded by $ O(\sum_j (pn_j/n-1)^4)$, which is in turn bounded trivially by 
$$(\sum_j (pn_j/n-1)^2)^2 = O(\frac{p^2b \log n}{n} \sum_j (pn_j/n-1)^2) = O(p^{-1-\eps/4} \sum_j (pn_j/n-1)^2),$$ 
and hence are negligible when $p \le n^{1/3(1+2\eps)}$ (see also \eqref{eqn:small:factor}). So we obtain
\begin{equation}\label{eqn:Stirling:final}
\binom{n}{n_0,\dots, n_{p-1}}\frac{p^{(d-1)n} \prod_{j=0}^{p-1} (d n_j)!}{(dn)!} =(1+o(1)) d^{(p-1)/2} \exp\Big(\frac{(d-1)pn}{2} \sum_j (\frac{n_j}{n}-\frac{1}{p})^2 - \frac{(d-1)p^2n}{6} (\frac{n_j}{n}-\frac{1}{p})^3 \Big).
\end{equation} 

\subsection{Treatment of the characteristic function}


We notice that $|\phi_{X-\bmu}(\Bs)|=1$ iff 
$$\Bs \in 2\pi \Z^p + 2\pi (0,1/p,\dots, (p-1)/p)\Z + (1,\dots, 1) \R.$$
For $\kappa>0$, for $j=0,\dots, p-1$ we define the domains
$$B_j(\kappa) = 2\pi j (0,1/p,\dots, (p-1)/p) + Q( \{\Bx \in \R^{p-1}: \|\Bx\|^2 \le \kappa \} \times [0, 2 \sqrt{p}\pi]),$$
where $Q$ is an orthogonal transform of the form $Q = [O, \1/\sqrt{p}]$ and $O$ is an orthogonal transform in the space $\1^\perp$.

Suppose that $\Bs \in B_j(\kappa)$ for some $j$, and $d=3$. Then $\Bs = 2\pi j (0,1/p,\dots, (p-1)/p)+O\Bx +y\1$ for some $\|\Bx\|^2 \le \kappa$ and $y\in [0, 2 \sqrt{p}\pi]$. Let $s'=Ox$.
\begin{align*}
& |\phi_X(\Bs)|  = |\frac{1}{p^2} \sum_{a,b} e^{i (s'_a+s'_b+s'_{-(a+b)})}| \\
& = \frac{1}{p^2} \sum_{a,b} e^{i (s'_a+s'_b+s'_{-(a+b)})}  = 1 - O(\frac{1}{p^2}\sum_{a,b} \|\frac{s'_a+s'_b+s'_{-(a+b)}}{2\pi}\|_{\R/\Z}^2)  \\
& = 1 - O(\frac{1}{p^2} p \sum_{a} \|\frac{s'_a}{2\pi}\|_{\R/\Z}^2)= 1- O(\kappa/p).
\end{align*}

Our main result says the converse.



\begin{theorem}[Inverse result for fixed $d$]\label{thm:inverse:d} Assume that
$$|\phi_{X-\bmu} (\Bs)| \ge 1 - \al p^{-2}$$
where $\al$ is a small constant. Then there exists $j$ such that $\Bs \in B_j(\kappa)$ for some $\kappa \le Ap^{-1}$, where $A$ is a sufficiently large constant depending on $\al,d$.
\end{theorem}
This is an improvement of \cite[Proposition 2.3]{H2} as there the right hand side is replaced by $1- O(p^{-3})$. Compared to \cite{H2}, our proof for this new setting is more complicated, but we believe that this is a delicate matter. In application, $\kappa$ is set to be $\delta$.

Let 
$$S= \frac{1}{p^{p-1}} \sum_{s_0,\dots, s_{p-1} \in \F_p, \sum_i s_i =0} e_p(-\sum_k d s_k n_k)  \times \Big[(\frac{1}{p})^{d-1} \sum_{(k_1,\dots,k_d) \in [p]^d, \sum_i k_i =0  } e_p(\sum_i s_{k_i} )\Big]^n.$$ 
In what follows we will be assuming the results of Lemma \ref{thm:inverse:d}. Our next goal has two steps
\begin{itemize}
\item (Approximation 1). In the first step we show that $S$ can be approximated by a function $f(\sum_j \mfn_j \log \mfn_j )$ of $\sum_j \mfn_j \log \mfn_j $. 
\vskip .1in
\item (Approximation 2). In the second step we combine with the Stirling part to estimate 
$$ \sum_{(n_0,\dots, n_{p-1})} d^{(p-1)/2} \exp((d-1)n\sum_j \mfn_j \log \mfn_j + \log p) f(\sum_j \mfn_j \log \mfn_j ).$$
 
\end{itemize}

\subsection{Step 1} For equidistributed $(n_0,\dots, n_{p-1})$ we first write
\begin{align}\label{eqn:start}
\P(X_1+\dots+ X_n = (dn_0,\dots, dn_{p-1})) & = \frac{1}{(2\pi)^p} \int_{2\pi \R^p/\Z^p} \phi_{X-\bmu}^n(\Bx) e^{-i\lang \Bx, d\Bn - n \bmu \rang} d\Bx \nonumber \\
& =  \frac{p^{3/2}}{(2\pi)^{p-1}}\int_{\Bx \in \R^{p-1}: \|\Bx\|_2^2 \le \delta} \phi_{X-\bmu}^n(O\Bx) e^{-i \lang O\Bx, d\Bn - b\bmu\rang } + O(e^{-\al n/p^{2}}).
\end{align}
where the error term $O(e^{-\al n/p^2})$ comes from $|\phi_{X-\bmu}(\Bx)| \le 1 -\al/p^2$ and Theorem \ref{thm:inverse:d}.

We write 
\begin{align*}
\phi_{X -\bmu}(O\Bx) = \E (1+ i\lang  O\Bx, X-\bmu\rang - \frac{1}{2} \lang O\Bx, X-\bmu\rang^2 - \frac{i}{6} \lang O\Bx, X-\bmu\rang^3 +O(\lang O\Bx, X-\bmu\rang^4) ).
\end{align*}

Let 
$$\Bs := O\Bx.$$ 
Then clearly $\|\Bs\|_2 = \|\Bx\|_2$. Because the columns of $O$ are orthogonal to $\1$, we have $\sum_i s_i =0$. For $\Phi(\Ba) \in \CU_{d,p}$ 
$$\lang \Bs, \Phi(\Ba)-\bmu \rang =  s_{a_1}+\dots+ s_{a_{d-1}} + s_{-\sum a_i}.$$
From the above discussion it suffices to work with 
$$\|s\|_2^2 \le \delta.$$

The first moment is zero as $\sum_i s_i =0$,
$$\E \lang \Bs, X-\bmu \rang = \frac{1}{p^{d-1}} \sum_{a_1,\dots, a_{d-1}} (\sum_i s_{a_i}+ s_{-\sum_i a_i})=0.$$
 
For the second moment, 
$$\E \lang \Bs, X-\bmu \rang^2=  \frac{1}{p^{d-1}}\sum_{a_1,\dots, a_{d-1}, a_d, \sum_i a_i=0}  (\sum_i s_{a_i})^2 = \frac{1}{p^{d-1}} p^{d-2} \times d \sum_a s_a^2= (d/p) \|\Bs\|_2^2 .$$
For the third moment, we see that the sum is a multiple of $ \frac{1}{p^{d-1}} p^{d-2}\sum_a(s_a/p)^3$. 

$$\E \lang \Bs, X-\bmu \rang^3=  \frac{1}{p^{d-1}}\sum_{a_1,\dots, a_d, \sum_i a_i=0}  (\sum_i s_{a_i})^3 = \frac{1}{p^{d-1}} \sum_{a_1,\dots, a_d, \sum_i a_i=0} \sum_{1\le i_1, i_2,i_3 \le d} s_{a_{i_1}} s_{a_{i_2}} s_{a_{i_3}}$$ 

$$ \leq \frac{1}{p^{d-1}} [\sum_{a} s_a^3 d p^{d-2} (1+\frac{d(d-1)}{p} +\frac{d(d-1)(d-2)}{p^2}) +  \sum_{a\neq b} s_a^2 s_b 3 d(d-1)p^{d-3}(1+\frac{d-3}{p}) +  \sum_{a<b<c} s_a s_b s_c 6 \binom{d}{3} p^{d-4}].$$

By passing the sum $ \sum_{a<b<c} s_a s_b s_c$ to $(\sum_a s_b)^3$ and $ \sum_{a<b} s_a^2 s_b$ to $(\sum_a s_a^2) (\sum_a s_a)$, we arrive at 
$$\E \lang \Bs, X-\bmu \rang^3 = (\frac{d}{p} + \frac{c'}{p^2} + \frac{c''}{p^3})\sum_{i} s_i^3=: \frac{C_p}{p} \sum_{i} s_i^3$$
for absolute constants $c',c''$, where
$$C_p:= d + c'/p + c''/p^2.$$
Notice that we can bound $\sum_i |s_i|^3$ from above by $(\sum_i s_i^2)^{3/2}$, but this does not give us a desirable bound.

For the fourth moment,
\begin{align*}
\frac{1}{p^{d-1}}\sum_{a_1,\dots, a_d, \sum_i a_i=0}  (\sum_i s_{a_i})^4 &= C_d \frac{1}{p^{d-1}}p^{d-3} \sum_{a,b} s_a^2s_b^2 +(d/p)\sum_a (s_a)^4\\
&= (C_d/p^{2}) \|\Bs\|_2^4 + (d/p)\sum_a (s_a)^4.
\end{align*}
Hence if $\|\Bs\|_2^2 \le \delta$ then 
$$n \|\Bs\|_2^4/p^2\le n \delta^2 /p^2 \le p^{-1-\eps/4}.$$
Also
$$\sum_a (s_a)^4 \le (\sum_a s_a^2)^2 \le \delta \sum_a s_a^2 \le p^{-1-\eps/4}\sum_a s_a^2.$$
Note that 
\begin{equation}\label{eqn:small:factor}
(1+O(1/p^{1+\eps/4}))^p = 1+o(1). 
\end{equation}
Hence for \eqref{eqn:start} it boils down to consider 
$$(1+o(1))  \frac{p^{3/2}}{(2\pi)^{p-1}}\int_{\Bs \in \R^{p-1}: \|\Bx\|_2^2 \le \delta} e^{-\frac{dn}{2p}\|\Bx\|_2^2} e^{-i \lang \Bs, d\Bn - b\bmu\rang + i \frac{C_pn}{p} \sum_a s_a^3 }.$$
In fact, we can extend the integral to all of $\R^{p-1}$  excluding $B_2(\delta) = \{\Bx: \sum_a s_a^2 \le \delta\}$ above because
\begin{align*}
& |(1+o(1))  \frac{p^{3/2}}{(2\pi)^{p-1}}\int_{\Bs \in \R^{p-1}: \R_{p-1} \bs B_2} e^{-\frac{dn}{2p}\|\bs\|_2^2} e^{-i \lang \Bs, d\Bn - \bmu\rang + i \frac{C_pn}{p} \sum_a s_a^3 }|\\
& \le (1+o(1))  \frac{p^{3/2}}{(2\pi)^{p-1}}\int_{\|\Bs\|_2^2 \ge \delta} e^{-\frac{dn}{4p}\|\Bs\|_2^2} \le e^{-n\delta /8p} \le e^{-p^{1+\eps}}.
\end{align*}

Hence we can pass to $\R^{p-1}$, which factors out as
$$(1+o(1))  \frac{p^{3/2}}{(2\pi)^{p-1}} \prod_{j=0}^{p-1}\int_{\R}   e^{-\frac{dn}{2p}s^2} e^{-i t_js + i \frac{C_pn}{p}s^3}ds,$$
where 
 \begin{equation}\label{eqn:t_j}
 t_j :=  (d\frac{\Bn}{n} - \bmu)_j.
\end{equation}

\begin{lemma}\label{lemma:cube} We have 
$$|\prod_j \int_{\R}   e^{-\frac{dn}{2p}s^2} e^{-i n t_js + i C_p \frac{n}{p}s^3} ds| \le (1+o(1)) p^{3/2} (\frac{p}{2\pi d n})^{\frac{p-1}{2}} e^{-\frac{np}{2d} \sum_j t_j^2 + C_p np^2 \frac{1}{d^{3}} t_j^3 }.$$
\end{lemma}
Note that if we change variables to $r_j = \sqrt{pn/d} t_j$, then the above can be written as 
$$-\frac{np}{2d} \sum_j t_j^2 + C_p np^2 \frac{1}{d^{3}} t_j^3 = -\sum_j r_j^2/2+O(r_j^3 \sqrt{p/n}).$$
\begin{proof}
In what follows we will work with $\int_{\R}   e^{-\frac{dn}{2p}s^2} e^{-i n t_js + i \frac{C_pn}{p}s^3}ds$. We notice that 
$$\sum_j t_j^2  = \|d\frac{\Bn}{n} - \bmu\|_2^2 \le \frac{b \log n}{n}.$$
By the change of variable $y = \sqrt{\frac{dn}{p}} s$ we obtain $\sqrt{\frac{p}{dn}}  \int_{}   e^{-\frac{1}{2}y^2 +} e^{-i t_j  \sqrt{\frac{ np}{d}} y  + i C_p (\sqrt{\frac{p}{n}}) \frac{1}{d^{3/2}} y^3} dy$.

Let $t = t_j  \sqrt{\frac{ np}{d}} $ and $s= C_p (\sqrt{\frac{p}{n}}) \frac{1}{d^{3/2}}$, then $s$ is very small compared to $t$ and we are working with 
$$ \int_{\R}   e^{-\frac{1}{2}y^2} e^{-i t y  + i s y^3} dy =  e^{-\frac{1}{2}t^2}\int_{\R}   e^{-\frac{1}{2}(y+it)^2} e^{i s y^3} dy =   e^{-\frac{1}{2}t^2}\int_{\R+ it }   e^{-\frac{1}{2}z^2} e^{i s (z-it)^3} dz$$
$$= e^{-\frac{1}{2}(t^2-st^3)}\int_{\R+ it }   e^{-\frac{1}{2}z^2} e^{i s (z^3 - 3 z^2 (it) + 3z (it)^2)} dz= e^{-\frac{1}{2}(t^2-st^3)}\int_{\R+ it }   e^{-\frac{1}{2}z^2} e^{i sz^3 +3 st z^2 - 3 i s t^2 z} dz$$
$$= e^{-\frac{1}{2}(t^2-st^3)}\int_{\R+ it }   e^{-(\frac{1}{2}-3st)z^2}  e^{i (sz^3- 3 st^2 z)} dz=e^{-\frac{1}{2}(t^2-st^3)}\int_{\R}   e^{-(\frac{1}{2}-3st)x^2}  e^{i (sx^3- 3 st^2 x)} dx$$
by contour integral.

The integral can be bounded by $\sqrt{1+O(st)} \le e^{O(st)}$. Hence we have 
$$|\int_{\R}   e^{-\frac{1}{2}y^2} e^{-i t y  + i s y^3} dy|  \le  e^{-\frac{1}{2}(t^2-st^3) + O(st)}.$$
Notice that 
$$e^{s  \sqrt{np/d} \sum_j t_j } \le e^{s \sqrt{np/d} \sqrt{p} \sqrt{\sum_j t_j^2}} \le e^{s p \sqrt{\log n}}= o(1).$$  
Putting these bounds together, 
\begin{align*}
&|\prod_j \int_{\R}   e^{-\frac{dn}{2p}s^2} e^{-i n t_js + i C_p \frac{n}{p}s^3} ds| \le (1+o(1)) p^{3/2} (\frac{p}{2\pi d n})^{\frac{p-1}{2}} e^{-\frac{np}{2d} \sum_j t_j^2 + C_p (\sqrt{\frac{p}{n}}) \frac{1}{d^{3/2}} ( \sqrt{\frac{ np}{d}})^3  t_j^3 }\\
&\le (1+o(1)) p^{3/2} (\frac{p}{2\pi d n})^{\frac{p-1}{2}} e^{-\frac{np}{2d} \sum_j t_j^2 + C_p np^2 \frac{1}{d^{3}} t_j^3 }.
\end{align*}

\end{proof}
As a consequence of \eqref{eqn:start} and Lemma \ref{lemma:cube}, we thus obtain
\begin{equation}\label{eqn:step1}
P(X_1+\dots+ X_n = (dn_0,\dots, dn_{p-1})) \le (1+o(1)) p^{3/2} (\frac{p}{2\pi d n})^{\frac{p-1}{2}} e^{-\frac{np}{2d} \sum_j t_j^2 + C_p np^2 \frac{1}{d^{3}} t_j^3 },
\end{equation}
where $t_j$ are defined in \eqref{eqn:t_j}.

\begin{remark}\label{remark:=} We note that it is possible to show that $\int_{\R}   e^{-(\frac{1}{2}-3st)x^2}  e^{i (sx^3- 3 st^2 x)} dx$ is actually $1+ O(st)$ (rather than bounded by this amount), but we won't need it here.
\end{remark}
\subsection{Step 2: Completion of proof of Proposition \ref{prop:clt:directed}} First recall that

\begin{align*}
& \sum_{\Bv\in S_{n_0,\dots, n_{p-1}}}\P(A_{n,d}\Bv=0) = \binom{n}{n_0,\dots, n_{p-1}}\frac{ p^{n(d-1)} \prod_{j=0}^{p-1} (d n_j)!}{(dn)!} \times \P(X_1+\dots+ X_n = (dn_0,\dots, dn_{p-1})).
\end{align*} 
Recalling the first factor from \eqref{eqn:Stirling:final} and the second factor from \eqref{eqn:step1}, after cancellation we obtain
$$(1+o(1)) p^{3/2} (p/2\pi n)^{(p-1)/2}e^{-(pn/2) \sum_j (\frac{n_j}{n}-\frac{1}{p})^2 - [\frac{(d-1)}{6} - \frac{C_p}{d^3}] np^2 (\frac{n_j}{n}-\frac{1}{p})^3}.$$
Our main goal in this part is the following (where $D_p= \frac{(d-1)}{6} - \frac{C_p}{d^3}$)
\begin{lemma}\label{lemma:sum:discrete} We have
$$\sum_{(n_0,\dots, n_{p-1})\in \CE, \sum_j j n_j \equiv 0 \pmod p}(1+o(1)) p^{3/2} (p/2\pi n)^{(p-1)/2}e^{-(pn/2) \sum_j (\frac{n_j}{n}-\frac{1}{p})^2 - D_p p^2n (\frac{n_j}{n}-\frac{1}{p})^3  } = 1+o(1).$$
\end{lemma}
It is clear that Proposition \ref{prop:clt:directed} then follows. For Lemma \ref{lemma:sum:discrete}, we first show that one can pass from $\sum j n_j \equiv 0 \pmod p$ to general $(n_0,\dots, n_{p-1})\in \CE$. 

\begin{claim} We have
\begin{align*}
& \sum_{(n_0,\dots, n_{p-1})\in \CE, \sum_j j n_j \equiv 0 \pmod p}p^{3/2} (p/2\pi n)^{(p-1)/2}e^{-(pn/2) \sum_j (\frac{n_j}{n}-\frac{1}{p})^2 - D_p p^2n(\frac{n_j}{n}-\frac{1}{p})^3  } \\
& =(1+o(1)) \sum_{(n_0,\dots, n_{p-1})\in \CE}p^{1/2} (p/2\pi n)^{(p-1)/2}e^{-(pn/2) \sum_j (\frac{n_j}{n}-\frac{1}{p})^2 - D_p p^2n (\frac{n_j}{n}-\frac{1}{p})^3}.
\end{align*}
\end{claim}

\begin{proof} First, it is clear from \cite[Eqn (3.14), (3.15)]{H2} that
$$e^{-(pn/2) \|(\frac{\Bn + \Be_k -\Be_0}{n} - \frac{\bmu}{d})\|_2^2} = (1+O(p \log^{1/2} n/n^{1/2})) e^{-(pn/2) \| (\frac{\Bn}{n} - \frac{\bmu}{d})\|_2^2}.$$
Note that
$$((\frac{\Bn + \Be_k -\Be_0}{n} - \frac{\bmu}{d})_j)^3  = ((\frac{\Bn}{n} - \frac{\bmu}{d})_j + (\frac{\Be_k -\Be_0}{n})_j)^3 =  ((\frac{\Bn}{n} - \frac{\bmu}{d})_j)^3 + O( ((\frac{\Bn}{n} - \frac{\bmu}{d})_j)^2/n) +  |(\frac{\Bn}{n} - \frac{\bmu}{d})_j|/n^2 + O(1/n^3)$$
and clearly $p^2 \sum_j ((\frac{\Bn}{n} - \frac{\bmu}{d})_j)^2 \le p^2 \log n /n$,
Hence we see that
\begin{align*}
 e^{-(pn/2) \sum_j (\frac{n_j}{n}-\frac{1}{p})^2 - D_p p^2n (\frac{n_j}{n}-\frac{1}{p})^3  }& = (1+O(p \log^{1/2} n/n^{1/2} + (p^2 \log n) /n) \times \\
& \times e^{-(pn/2)\sum_j((\frac{\Bn + \Be_k -\Be_0}{n} - \frac{\bmu}{d})_j)^2  - D_p p^2n \sum_j((\frac{\Bn + \Be_k -\Be_0}{n} - \frac{\bmu}{d})_j)^3}.
\end{align*}
Summing over $j$ and taking the average we obtain the claim.
\end{proof}

We then claim that 
$$\sum_{(n_0,\dots, n_{p-1})\in \CE}p^{1/2} (p/2\pi n)^{(p-1)/2}e^{-\frac{pn}{2} \sum_j (\frac{n_j}{n}-\frac{1}{p})^2 - D_p p^2n (\frac{n_j}{n}-\frac{1}{p})^3  } = 1+o(1).$$
Replacing this Riemann sum by its integral, it suffices to show that 

\begin{lemma}\label{lemma:identity:integral} With the choices of parameters as in \eqref{eqn:p1},
$$\int_{\|\By\|_2^2 \le p \log n} (1/\sqrt{2\pi})^p e^{-\sum_j y_j^2/2 + D_p \sqrt{p/n} \sum_j y_j^3} dy_1\dots dy_p \le  1+o(1).$$
\end{lemma}

\begin{proof} For each positive $R$ such that $R^2 \le p \log n$ we consider $\sum_i y_i^2 =R^2$ and write 
$$\int_{\|\By\|_2^2=R^2} e^{-R^2/2 +  D_p\sqrt{p/n} \sum_j y_j^3} dy_1\dots dy_p = e^{-R^2/2}R^p \int_{\|\Bx\|_2=1} e^{D_p \sqrt{p/n} R^3 \sum_j x_j^3} dx_1\dots dx_p.$$

It is well known that the uniform measure $\frac{1}{\Vol(S_p)} dx_1\dots dx_p$ over the unit sphere can be replaced by $x_j= \xi_i/ \sqrt{\sum_i \xi_i^2}$ where $\xi_1,\dots, \xi_p$ are iid standard Gaussian. As such, our first goal is to show that for $R^2 \le p \log n$, with respect to the random Gaussian variables $\xi_1,\dots, \xi_p$
\begin{equation}\label{eqn:gau:id1}
\E e^{D_p\sqrt{p/n}R^3  \sum_i (\xi_i/ \sqrt{\sum_i \xi_i^2})^3} = 1+o(1).
\end{equation}
First, as $R^3 \le (p \log n)^{3/2}$ and clearly $e^{-cp} e^{D_p\sqrt{p/n}R^3} =o(1)$ if $p \ll n^{1/3}$, by large deviation of $\sum_j \xi_j^2$ (that $\P(\sum_i \xi_i^2 < p/4 \mbox{ or } \sum_i \xi_i^2  >4p) \le e^{-cp}$ for some absolute constant $c$), the contribution in the expectation when $\sum_i \xi_i^2 < p/4$ or $\sum_i \xi_i^2>4p$ is $o(1)$. Let $\CE_{b}$ denote the event $p/4 \le \sum_i \xi_i^2 \le 4p$.

Second, on the event $R^3 \sum_i (\xi_i/ \sqrt{\sum_i \xi_i^2})^3 \le p$, as $\sqrt{p/n} \le 1/p^{1+\eps/8}$ we must have $e^{D_p\sqrt{p/n} R^3 \sum_i (\xi_i/ \sqrt{\sum_i \xi_i^2})^3} = e^{o(1)}=1+o(1)$. Hence it remains to focus on the events $p^{1+\eps/8} \le R^3 \sum_i (\xi_i/ \sqrt{\sum_i \xi_i^2})^3$ and the event $\CE_{b}$ that $\sum_i \xi_i^2$ has order $p$.

\begin{claim} For $p^{1+\eps/8} \le t \le R^3$ we have 
$$\P\Big(R^3 \sum_i (\xi_i/ \sqrt{\sum_i \xi_i^2})^3 \ge t \wedge \CE_{b}\Big) \le e^{- c t^{2/3}p/R^2},$$
for some absolute constant $c$.
\end{claim}

\begin{proof} For short, let $\al:=  t p^{3/2}/R^3$. As $\P(\xi_i^3 \ge x) = \P(\xi \ge x^{1/3}) =O(e^{- x^{2/3}/2})$ if $x$ is large, by a result of Nagaev (see for instance \cite[Eqn. (1.2) and Theorem 1]{GRR}) we have
$$\P(\sum_{i=1}^p \xi_i^3 \ge \al)  = \P(\sum_i \xi_i^3/p \ge \al/p)  \le e^{-c p^{2/3} (\al/p)^{2/3}} = e^{-c\al^{2/3}},$$
for some absolute constant $c$.
\end{proof}
Back to our proof, with $X = R^3 \sum_i (\xi_i/ \sqrt{\sum_i \xi_i^2})^3$,
\begin{align*}
& \E e^{D_p \sqrt{p/n}R^2 X}1_{p^{1+\eps/8} \le X \le R^3 \wedge \CE_{b}} \le \int_{p^{1+\eps/8}}^{R^3} \sqrt{p/n}  e^{D_p\sqrt{p/n}t} \P(X>t \wedge \CE_b)dt\\
&\le  \int_{p^{1+\eps/8}}^{R^3} \sqrt{p/n} R^2  e^{D_p\sqrt{p/n} t - c  t^{2/3}p/R^2}\le \int_{p^{1+\eps/8}}^{R^3} \sqrt{p/n} e^{ -(c/2) \ t^{2/3}p/R^2} =o(1)
\end{align*}
where in the second to last bound we used the fact that $t\le R^3$ and $R^2 \le p \log n$ and the choices of parameters from \eqref{eqn:p1} (where we note that our bounds are slightly better than needed). With this we are done with proving \eqref{eqn:gau:id1}.
\end{proof}

We have thus shown that for each $R$ so that $R^2 \le p \log n$
\begin{align*}
\int_{\|\By\|_2^2=R^2} e^{-R^2/2 +  D_p\sqrt{p/n} \sum_j y_j^3} dy_1\dots dy_p &= e^{-R^2/2}R^p \int_{\|\Bx\|_2=1} e^{D_p \sqrt{p/n} R^3 \sum_j x_j^3} dx_1\dots dx_p \\
&= e^{-R^2/2}R^p \Vol(S_{p}) \E e^{D_p\sqrt{p/n}R^3  \sum_i (\xi_i/ \sqrt{\sum_i \xi_i^2})^3} \\
& = (1+o(1))e^{-R^2/2}R^p \Vol(S_{p}).
\end{align*}

Hence 
\begin{align*}\int_{\|\By\|_2^2 \le p \log n} (1/\sqrt{2\pi})^p e^{-\sum_j y_j^2/2 + D_p\sqrt{p/n} \sum_j y_j^3} dy_1\dots dy_p &= (1+o(1))\int_{R\le \sqrt{ p \log n}} (1/\sqrt{2\pi})^p e^{-R^2/2}R^p \Vol(S_{p})dR \\
&=1+o(1),
\end{align*}
completing the proof of Lemma \ref{lemma:identity:integral}.

\section{Proof of Theorem \ref{thm:inverse:d}}\label{section:inverse}

For more generality, we will choose $\eta$ so that
\begin{equation}\label{eqn:eta}
\eta^2 p \asymp \al p^{-1}
\end{equation}
and assume that 
$$|\frac{1}{p^{d-1}}\sum_{a_1,\dots, a_{d-1}} \exp(i(s_{a_1}+\dots+s_{a_{d-1}} + s_{-\sum_i a_i}))| \ge 1 - \eta^2.$$
In other words, there exists $x_0$ such that 
\begin{equation}\label{eqn:eta-x}
\frac{1}{p^{d-1}}\sum_{a_1,\dots, a_{d-1}} \Re(\exp(i(s_{a_1}+\dots+s_{a_{d-1}} + s_{-\sum_i a_i}+x_0))) \ge 1 - \eta^2.
\end{equation}
By shifting every $s_a$ by a constant, we can assume 
\begin{equation}\label{eqn:s_01}
s_0=0.
\end{equation}
Note that $|\sin(x)| \ge 2\|x/\pi\|_{\R/\Z}$ (which we replace by $\|.\|$ for short),
$$\Re( \exp(i(s_{a_1}+\dots+s_{a_{d-1}} + s_{-\sum_i a_i}+x_0))) = \cos(s_{a_1}+\dots +s_{a_{d-1}}+  s_{-\sum_i a_i}+x_0)$$
$$= 1- 2\sin^2(\frac{s_{a_1}+\dots +s_{a_{d-1}}+  s_{-\sum_i a_i}+x_0}{2}) \le 1- 8 \|\frac{s_{a_1}+\dots +s_{a_{d-1}}+  s_{-\sum_i a_i}+x_0}{2\pi}\|^2.$$ 
Hence the assumption of Theorem \ref{thm:inverse:d} (or more specifically \eqref{eqn:eta-x}) implies
\begin{equation}\label{eqn:avr:d}
\sum_{a_1,\dots, a_{d-1}}  \|\frac{s_{a_1}+\dots +s_{a_{d-1}}+  s_{-\sum_i a_i}+x_0}{2\pi}\|^2 \le \eta^2 p^{d-1}/8.
\end{equation}

{\bf Macroscopic analysis.} Our first goal is the following

\begin{lemma}\label{lemma:macro} There exists $d_0 \in \{0,\dots, p-1\}$ such that for all $a$
$$\|\frac{s_a}{2 \pi} - \frac{d_0 a}{p}\| \ll \sqrt{ \eta^2 p}.$$
\end{lemma}
For this, we first show the following 
\begin{claim}\label{claim:d} We have 
\begin{enumerate}
\item \begin{equation}\label{eqn:x_0:d}
\|\frac{x_0}{2 \pi}\| \ll_d \sqrt{\eta^2 p}.
\end{equation}
\item Also, for all $a_1,\dots, a_{d-1}$
$$\|\frac{s_{a_1}+\dots+ s_{a_{d-1}} + s_{-a_1-\dots - a_{d-1}}}{2 \pi}\| \ll_d \sqrt{\eta^2 p}.$$
\end{enumerate}
\end{claim}
Note that the proof of this result is similar to the first part of the proof of \cite[Proposition 2.3]{H2}.
\begin{proof} We have learned that 
$$\sum_{\Ba}  \|\frac{ \lang \Bs, \Ba \rang+x_0}{2 \pi}\|^2 \le \eta^2 p^{d-1}/8.$$
Let $\eps_0$ be sufficiently small (such as $\eps_0 =1/100$). Let $\CG$ be the set of $\Ba$ (such that $\sum_i a_i =0$) where  $\|\frac{ \lang \Bs, \Ba \rang+x_0}{2 \pi}\| \le \sqrt{\eps_0^{-1} \eta^2 p}$, then we have 
$$|\CG| \ge (1-\eps_0/8p) p^{d-1}.$$
Fix $\Ba_1 =\Ba=(a_1,\dots, a_d)$ with $\sum_i a_i = 0$, and let $\Bw = \Phi(\Ba_1)$. The total number of zero sum $d\times d$ matrix (of zero column and row sums) with the first row $\Ba_1$ is $p^{(d-1)(d-2)}$. For any $\Bb$, the number of such matrices with first row $\Ba_1$ and some other row  $\Bb$ is at most $(d-1) p^{(d-3)(d-1)}$, and with first row $\Ba_1$ and some other column  $\Bb$ is at most $d p^{(d-2)(d-2)}$. So the number of zero sum $d\times d$ matrix with the first row $\Ba_1$ and at least one row or column belonging to the set complement $\bar{\CG}$ is bounded by $((d-1) p^{(d-3)(d-1)} + d p^{(d-2)(d-2)}) |\bar{\CG}| \le  2\eps_0 d p^{(d-3)(d-1)} < p^{(d-3)(d-1)}$. It thus follows that there exists a zero sum $d\times d$ matrix with the first row $\Ba_1$ and all other rows $\Ba_2,\dots, \Ba_d$ and columns $\Bb_1,\dots, \Bb_d$ belonging to $\CG$. By the triangle inequality,

$$\|\frac{ \lang \Bs, \Ba \rang+x_0}{2 \pi}\|  =\|  \frac{ \sum_{i=1}^n (\lang \Bs, \Bb_i \rang+x_0) - \sum_{j=2}^n (\lang \Bs, \Ba_j \rang+x_0)}{2 \pi}\| \le (2d-1)  \sqrt{\eps_0^{-1} \eta^2 p}.$$
 
\end{proof}
Choosing $a_1=a, a_2=-a, a_3=\dots= a_{d-1}=0$, we obtain that 
$$\|\frac{s_{a}+ s_{-a}}{2 \pi}\| \ll \sqrt{\eta^2 p}.$$
Hence without loss of generality we can assume that $s_{-a}=-s_a$.

\begin{proof}(of Lemma \ref{lemma:macro}) For short we let 
$$q := \lceil  p \sqrt{\eta^2 p} \rceil .$$
Note that by definition $q$ is much smaller than $p$. By Claim \ref{claim:d}, given that $\al$ is sufficiently small, we have that 
$$\|\frac{s_{a_1}+\dots+ s_{a_{d-1}} + s_{-a_1-\dots - a_{d-1}}}{2\pi}\| <\sqrt{\eta^2 p}, \forall a_1,\dots, a_{d-1}.$$
Let $K$ be a sufficiently large even constant (and recall that $p$ is sufficiently large). It suffices to assume $s_a \in [-\pi,\pi]$ for all $a$. We first choose $k_a\in \Z/p\Z$ such that 
$$|\frac{k_a}{p} - \frac{s_a}{2\pi}| \le \frac{1}{2p}.$$
Our goal is to show that there exists an integer $d_0$ such that 
\begin{equation}\label{eqn:k_a:linear}
k_a\equiv d_0 a + [-5Kq  -5Kq]\pmod p, \mbox{ for all } a.
\end{equation}
Lemma \ref{lemma:macro} would then follow because
$$\|\frac{s_a}{2\pi} - \frac{d_0a}{p}\|_{\R/\Z} \le \frac{10Kq+1}{2p} \ll \sqrt{\eta^2 p}.$$

In what follows we show \eqref{eqn:k_a:linear}. Consider intervals (arcs) $I_a$ in $\Z/p\Z $ of length $Kq$ centered at $k_a$,
$$ I_a = [k_a-K q /2 ,k_a+Kq/2] \subset \Z/p\Z.$$ 
In what follows $K$ is a sufficiently large constant. Note that $I_0= [-Kq/2, Kq/2]$. Let $B$ be the set of the following points in $\Z/p\Z \times \Z/p\Z$,
$$B=\{(a,l), a\in \Z/p\Z,  l \in I_a\}.$$
As $k_{-a}= -k_a$, this is a symmetric set.

For each $k\ge 1$, we will be interested in the set $kB:=\{b_1+\dots+b_k, b_i \in B\}$. In particular,
$$(d-1)B = \bigcup_{a_1,\dots, a_{d-1}} \{a_1 + \cdots + a_{d-1}\} \times (I_{a_1} +\dots+ I_{a_{d-1}}).$$
For this set, on the one hand,
$$I_{a_1} + \dots + I_{a_{d-1}} = [\sum_i k_{a_i} -Kq(d-1)/2 , \sum_i k_{a_i} + Kq(d-1)/2].$$
On the other hand, by definition, $\|\frac{s_{a_1}+\dots+ s_{a_{d-1}} + s_{-a_1-\dots - a_{d-1}}}{p}\|   \le \sqrt{\eta^2 p}$, and so by the triangle inequality
$$|\frac{k_{a_1} +\dots+ k_{a_{d-1}} + k_{-(a_1+\dots+a_{d-1})}}{p}| \le \|\frac{s_{a_1}+\dots+ s_{a_{d-1}} + s_{-a_1-\dots - a_{d-1}}}{p}\| + \frac{d}{2p} \le \sqrt{\eta^2 p} + \frac{d}{2p}.$$
Hence (noting the choice of $\eta$) we have
$$|k_{a_1} +\dots+ k_{a_{d-1}} + k_{-(a_1+\dots+a_{d-1})}| \le 2 p \sqrt{\eta^2 p} \le 2q.$$

It thus follows that, with $c=-(a_1+\dots+a_{d-1})$, over $\Z/p\Z$  
\begin{align}\label{eqn:containment}
I_{a_1} +\dots+ I_{a_{d-1}} &= [\sum_i k_{a_i} -K q(d-1)/2 , \sum_i k_{a_i} + Kq(d-1)/2] \nonumber \\ 
& \subset   [- k_{-(a_1+\dots+a_{d-1})} -Kq(d-1)/2 - 2q, - k_{-(a_1+\dots+a_{d-1})} + Kq(d-1)/2+2q]\nonumber \\ 
& \subset -I_{c} +[-\frac{Kq(d-1)}{2} -2q, \frac{Kq(d-1)}{2} + 2q].
\end{align}
Notice that the set $B$ has size $(Kq+1)p$, while the union of the sets $\{c\} \times (-I_c +[-\frac{Kq(d-1)}{2} -2q, \frac{Kq(d-1)}{2} + 2q)])$ has size $p(Kq(d-1)+4q$. 
Thus we have 
\begin{equation}\label{eqn:growth}
|(d-1)B| \le (d-1)Kq p + 4pq \le (d-1)|B| + 4pq.
\end{equation}
Another pleasant property following from \eqref{eqn:containment} is that for any $k\ge 0$
$$I_{a_1} +\dots+ I_{a_{d-1} + \dots+ I_{a_{d-1+k}}} \subset  I_{a_{d}} + \dots+ I_{a_{d-1+k}} -I_{c} +[-\frac{Kq(d-1)}{2} -2q, \frac{Kq (d-1)}{2} + 2q].$$
Notice that the size of the union of the sets $\{a_{d}+\dots+ a_{d-1+k}\} \times (I_{a_{d}} + \dots+ I_{a_{d-1+k}} -I_{c} +[-\frac{Kq (d-1)}{2} -2q, \frac{Kq(d-1)}{2} + 2q])$ is bounded by $|kB-B|+p(K q (d-1)+4q)$. Note that as $B$ is symmetric (that is, $-B=B$) we have
\begin{equation}\label{eqn:growth'}
|(d-1+k)B| \le |(k+1)B| + (d-1)|B| + 4pq.
\end{equation}
Note that we allow $k+1>d-1$, and hence can repeat the process.

{\bf {When $d=3$}}. We have 
$$|2B| \le 2|B| + 4pq.$$
Note that when $K$ is large, $4pq$ is small compared to $|B|$. This is similar to Freiman's $(3n-3)$-theorem \cite{TVbook} except that our setting is not torsion-free. We then use  a very recent result by Lev \cite[Theorem 1]{Lev}, which says that if $B$ is not contained in fewer than $\ell$ cosets of a subgroup of $G=\Z/p\Z \times \Z/p\Z$ and if $|2B| \le 3(1-1/\ell) |B|$, then there exists an arithmetic progression $P \subset G$ of size $|P|\ge 3$ and a subgroup $G'$ of $G$ such that 
\begin{equation}\label{eqn:containment:1}
|P+G'| = |P| |G'|, B \subset P+G', \mbox{ and } (|P|-1)|G'| \le |2B| -|B|.
\end{equation}

For short, we call such structure $P+G'$ {\it coset progression (of rank one)}. We will choose $\ell=4$. Consider the case that $B$ is contained in 3 cosets of a subgroup $G'$ of $G$. By definition, $G'$ must be $\Z/p\Z \times \{0\}$. However this is impossible because $|I_a|=K q +1>3$ (for any $a$) as $K$ is large.

Hence $B$ cannot be contained in 3 cosets, as $|B+B| \le |2B| + 4p < 3(1-1/4)|B|$, we see that there is some subgroup $G'$ and some arithmetic progression $P \subset G$ such that
$$B \subset P+ G'.$$
We then divide into several subcases.
\begin{enumerate}
    \item[(i)]$G'=\{0\} \times \Z/p\Z$, as $I_a$ is a proper subset of $\Z/p\Z$, this is impossible.
    \vskip .05in

    \item[(ii)] $G' = \{0\} \times \{0\}$, we then see that $B\subset P$ and $|P|$ has size at most $|2B|-|B| + 1 \le |B|+4pq+ 1=(K+5)pq + 1$. As $P$ is an arithmetic progression, it can be written as $P= \{(p_0,q_0)+ i (x,y), 0\le i \le |P|-1 \}$ for some $(p_0,q_0)$ and $(x,y)$ in $G$, where it is clear that $x\neq 0$. For each $a$, consider $S_a=\{0 \le i\le |P|-1, p_0+ix=a\}$. Each $i\in S_a$ has the form $i=i_a + lp$ for some representative $i_a$. So $I_a \subset \{q_0 + (i_a+lp)y\} = \{q_0+i_ay\}$. However, this is impossible as $I_a$ has length $Kq +1$, which is sufficiently large.
\vskip .05in
\item[(iii)] $G'$ is a cyclic proper subgroup of form $\{i(g,h), 0\le i \le p-1\}$, for some $(g,h) \neq (0,0)$ in $G$. We see that  
$(|P| -1) p \le (K+5)pq$, and so $|P| \le (K+5)q$. Write $P=\{(p_0,q_0) +  j (x,y), 0\le j \le |P|-1 \le K+4\}$. For each $a$ we let $S_a$ be the set of pairs $(i,j)$ such that $p_0 + ig + jx =a$. Then it is clear that $g\neq 0$,  $I_a \subset \{q_0 +i h + j y , (i,j) \in S_a\}$, and 
$$i = g^{-1}(a-p_0) - j g^{-1} x.$$
So we have  
\begin{align*}
I_a &\subset \{q_0 +(g^{-1}(a-p_0) - j g^{-1} x) h + j y , 0\le j\le  (K+5)q\}\\
& = \{q_0 +g^{-1} h(a-p_0) -  j (g^{-1} x h -  y), 0\le j\le  (K+5)q\}.
\end{align*}
Hence either $g^{-1} x h -  y=-1$ or $g^{-1} x h -  y =1$. Without loss of generality we assume the latter. Note that as $I_0 \subset  \{q_0 -g^{-1} hp_0 -  j (g^{-1} x h -  y), 0\le j\le  (K+5)\sqrt{\eta^2 p} \}$, we must have (with some room to spare)
$$q_0 -g^{-1} hp_0 \in [-2Kq,2Kq].$$
Putting together, 
\begin{align*}
I_a = [k_a-K q /2 ,k_a+K q  /2] &\subset \{q_0 -g^{-1} hp_0 -  j (g^{-1} x h -  y) +g^{-1} ha, 0\le j\le (K+5)q\}\\
&\subset [-4Kq + g^{-1} ha, 4Kq+ g^{-1}ha].
\end{align*} 
We thus conclude that for all $a$ we have $k_a  \in [g^{-1}h a - 5K q , g^{-1}h a + 5K q]$, confirming \eqref{eqn:k_a:linear}.

\end{enumerate}

From the proof, we can actually obtain \eqref{eqn:k_a:linear} from a slightly more general result (when we applied \cite{Lev} for $\ell =4$ as above)

\begin{lemma}\label{lemma:P'}
Assume that $B=\{(a,l), a\in \Z/p\Z, l \in I_a\}$, where $I_a$ are intervals of length $K+1$ for sufficiently large $K$ as above, and $K \ll p$. Then if $|B + B| < 9|B|/4$, the set $B$ can be contained in a coset progression of rank one $P+G'$ as in \eqref{eqn:containment:1}.
\end{lemma}

\begin{corollary}\label{cor:general:h} Assume that for some positive integer $h$ of order $O(1)$ of we have 
$$|2^h B| < (2.25-\eps)^h |B|.$$
Then there exists a coset progression of rank one $P+G'$ of size $O_h(|B|)$ as in \eqref{eqn:containment:1} such that $B \subset P+G'$. 
\end{corollary}
It is important to note that if there exists such $P+G'$, we can argue as (iii) in the above to then arrive at \eqref{eqn:k_a:linear} (with $H$ depending on $h$.)
\begin{proof} By the assumption, there exists $0\le h' \le h-1$ so that 
$$|2^{h'+1} B| \le (2.25-\eps)|2^{h'} B|.$$
We will then apply Lemma \ref{lemma:P'} to the set $2^{h'}B$ to contain it in a coset progression $P+G'$ of size at most $2 |2^{h'}B| \le (2.25-\eps)^h |B|$. Now as $B \subset 2^{h'}B$ (as $B$ contains $\{0\}$ \footnote{This assumption is not needed, as $2^{h'}B$ contains a translation of $B$}), we hence have a similar containment for $B$. 
\end{proof}

{\bf General $d$}. From \eqref{eqn:growth'}, by choosing $k= 2^{h} - (d-1)$, where say $d-1 \le 2^{h-5}$, we have
$$|2^hB| \le  |2^{h-5}B| + 2^h|B|.$$
So with $B'= 2^{h-5}B$, 
$$|32B'| \le |B'| + 2^h|B|.$$ 
It suffices to consider the case $(2.25)^2 |8B'| < |B'| + 2^h|B|$, because otherwise $|32B'|\le (2.25)^2 |8B'|$ and we can apply Corollary \ref{cor:general:h}. So $|8B'| < |B'| + (2^h/2.25) |B|$. Similarly we arrive at 
$$|2B'| < |B'| + (2^h/2.25^5) |B|.$$ 
So it passes to consider $|B'| \le (2^{h-5}+\eps)|B|$ and again we can use Corollary \ref{cor:general:h}.

\end{proof}

\subsection{Microscopic analysis.} With the help of Lemma \ref{lemma:macro}, by replacing $s_a$ by $s_a - 2\pi (d_0 a/p)$ for all $a$, we are thus free to replace $\|.\|$ (that is $\|.\|_{\R/\Z}$) by $|.|$ as all the numbers are sufficiently small. In the next step we establish the following key estimate.

\begin{lemma}[structure of triple]\label{lemma:norm:str:d}
Assume that for all $a$
$$|\frac{s_{a}}{2\pi}|=o(1)$$
$$|\frac{s_{a_1}+\dots+ s_{a_{d-1}} + s_{-a_1-\dots - a_{d-1}}+x_0}{2\pi}|^2 \le \eta^2 p^2.$$ 
Then there exists an absolute constant $A$ such that
$$\sum_{a\in \Z/p\Z} |\frac{s_a}{2\pi}|^2 \le A \eta^2 p.$$
\end{lemma}
\begin{proof}(of Lemma \ref{lemma:norm:str:d}) By shifting everything by $x_0/d$ (in $2\pi \R/\Z$), we can assume $x_0=0$.  In what follows $\eps_0>0$ is a sufficiently small constant, which can change depending on the situation.  

For transparency, we again consider the simple case first.

{\bf {When $d=3$}}. Let $\CB$ be the collection of $(a,b)$ where  
$$|\frac{s_{a}+s_{b} + s_{-a-b}}{2\pi}| +|\frac{s_{-a}+s_{-b} + s_{a+b}}{2\pi}|\ge \eta \eps_0^{-1}.$$
By assumption,  
$$|\CB| \le \eps_0^2 p^2.$$
Let $\CG$ be the complement of $\CB$ in $(\Z/p\Z)^2$. Hence for all $(a,b) \in \CG$ (and at the same time $(-a,-b)\in \CG$ and $(a,-a-b), (b,-a-b) \in \CG$) we have that 
\begin{equation}\label{eqn:good:s:3}
|\frac{s_{a}+s_{b} + s_{-a-b}}{2\pi}|, |\frac{s_{-a}+s_{-b} + s_{a+b}}{2\pi}| < \eps_0^{-1}\eta=:\eta'.
\end{equation}
Let $\eta'$ be of order $p^{-1}$.  As before, for each $a$ we let $k_a\in \Z$ be such that $-p/2 < k_a < p/2$ and that $|\frac{s_a}{2\pi} - \frac{k_a}{p}| \le \frac{1}{2p}$. Then by the assumption of Lemma \ref{lemma:norm:str:d}, $k_a =o(p)$, and also by \eqref{eqn:good:s:3}, as long as $(a,b)\in \CG$ we have 
$$k_a+k_b+k_{-a-b} \in \{-3,\dots,3\}.$$
Indeed, this is because 
$$\|\frac{k_a + k_b + k_{-a-b}}{p}\| \le \|\frac{s_{-a}+s_{-b} + s_{a+b}}{2\pi}\| + \frac{3}{2p} \le \frac{3}{p}.$$
As this holds for all $(a,b)\in \CG$ (which consist most of the pairs $(a,b)$), we guess that $k_a$ must be linear in $a$. This is very similar to our situation in the previous section, except that here we are working over $\Z$, and not all but almost all pairs $(a,b)$ have this property. To confirm this we prove directly.
\begin{claim}\label{claim:ka:small} $|k_a|$ is at most $O(1)$ for all but $O(\eps_0 p)$ indices $a \in \{0,\dots, p-1\}$.
\end{claim}
\begin{proof}
We say that $a$ is {\it good} if the number of pairs $(b,c) \in \CG$ such that $a+b =-c$ is at least $(1-\eps_0)p$. It is not hard to see that there are $(1-\eps_0)p$ such good indices. Assume that $a_0$ is such that $|k_{a_0}|$ is the largest among the good $a$. Without loss of generality we assume that $k_{a_0}$ is positive.  Consider $(b,c)\in \CG$ such that $a+b=-c$. There are $(1-\eps_0)p$ such pairs, and because most indices are good, there are $(1-2\eps_0)p$ pairs in which $b,c$ are good. Because $k_c \in -(k_{a_0} + k_b) + \{-3,\dots,3\}$, and because $k_{a_0}$ is maximal, the following holds: either $k_b$ is negative, or if $k_b$ is positive then it must be at most $2$. So we can decompose $\{0,\dots,p-1\}$ into two sets $\CP = \{a , k_a \geq 0\}$ and $\CN = \{a , k_a < 0\}$ for each case.

    Assume that $|\CN| \ge 10\eps_0 p$. Let $s\in \CN$, then either there exists $t\in \CN$ such that $s+t =-k_{a_0}+ O(1)$ or $s = -k_{a_0} + O(1)$. In either case, if $k_{a_0} \gg 1$ then we have $|s| \ge |k_{a_0}|/2 +O(1)$. Now the set of $a\in \CN$ where $|k_a|> k_{a_0}/2+3$ cannot be of size $\eps_0 p$ because otherwise we could then choose two elements $a_1,a_2$ so that $k_{-a_1-a_2}>k_{a_0}$, a contradiction. Also, the number of of $a\in \CN$ such that $|k_a| \le k_{a_0}/2 -2$ cannot be more than $\eps_0 p$ because then  $|k_{-a_0-a}|\ge k_{a_0}/2+3$, and we have learned that the number of such  is at most $\eps_0 p$. Putting this together, we see that the remaining set $\CN^\ast$ of $a$ such that $|k_a|<k_{a_0}/2-2$ has size at least $|\CN| - 2\eps_0 p$ which has order around $-k_{a_0}/2$. To this end, consider the set $\{ (a,a')\in \CG, a+a', a,a' \in \CN^\ast \}$. This set has size at least approximately $|\CN^\ast|$, and if $a''= -(a+a')$ then $k_{a''}$ is approximately $k_{a_0}$. So if there are many such $a''$ and therefore a pair of $(a_1'', a_2'') \in \CG$, we then see that $k_{-a_1''- a_2''}$ is approximately $2 k_{a_0}$, a contradiction. Hence $|\CN| < 10 \eps_0 p$, and therefore $|\CP| \ge (1-11 \eps_0)p$, completing the proof.
\end{proof}

Let $B$ define the set of all indices satisfying Claim \ref{claim:ka:small}. Assume that $B$ is a proper subset of $\Z/p\Z$. Let $c \in (\Z/p\Z) \bs B$, then as $|B| \ge (1-O(\eps_0))p$, there are $(1-O(\eps_0^2))p$ pairs $(a,b) \in B^2$ such that $a+b+c=0$. By assumption we have 
$$\sum_{a,b; -(a+b) \notin B}  \|\frac{s_{a}+s_{b} + s_{-a-b}}{2\pi}\|^2 \le \sum_{a,b}  \|\frac{s_{a}+s_{b} + s_{-a-b}}{2\pi}\|^2 \le \eta^2 p^2/8.$$
Hence
$$\sum_{c, c=-(a+b) \notin B}  \|\frac{-s_a -s_b + s_{c}}{2\pi}\|^2 \le \eta^2 (1-O(\eps_0^2))^{-1} p^2/8.$$
Notice that when $a \in B$, as $k_a = O(1)$, we have $\| \frac{s_a}{2\pi}\| = O(\frac{1}{p}) = O(\eta')$. So by the triangle inequality
$$\|\frac{s_{c}}{2\pi}\| \le  (\|\frac{-s_a -s_b + s_{c}}{2\pi}\| +O(\eta')).$$
Thus we have
$$\sum_{c, c=-(a+b) \notin B}  \|\frac{s_{c}}{2\pi}\|^2 \le \sum_{c, c=-(a+b) \notin B}  2(\|\frac{-s_a -s_b + s_{c}}{2\pi}\|^2 + O({\eta'}^2)) \le 2\eta^2 (1-O(\eps_0^2))^{-1} p^2/8 + 2p^2 \eta^2.$$
This implies that (recall that there are $(1-O(\eps_0^2))p$ pairs $(a,b) \in B^2$ such that $a+b+c=0$)
$$\sum_{c \notin B}  \|\frac{s_{c}}{2\pi}\|^2 \le O({\eta'}^2 p).$$
Altogether,
$$\sum_{c}  \|\frac{s_{c}}{2\pi}\|^2 = \sum_{c \notin B}  \|\frac{s_{c}}{2\pi}\|^2 + \sum_{c \in B}  \|\frac{s_{c}}{2\pi}\|^2  =  O({\eta'}^2 p).$$
This completes the proof of our result for $d=3$.

{\bf Treatment for general $d$}. The proof here will be similar to the case $d=3$, so we will be brief. Let $\CB$ be the collection of $(a_1,\dots, a_{d-1})$ where  
$$|\frac{s_{a_1}+\dots+s_{a_{d-1}} + s_{-\sum a_i}}{2\pi}| \ge \eta \eps_0^{-1}.$$
By assumption,  
$$|\CB| \le \eps_0^2 p^{d-1}.$$
Let $\CG$ be the complement of $\CB$ in $(\Z/p\Z)^{d-1}$. Hence for all $(a_1,\dots, a_{d-1}) \in \CG$ 
 we have that 
\begin{equation}\label{eqn:good:s:d}
|\frac{s_{a_1}+\dots+s_{a_{d-1}} + s_{-\sum a_i}}{2\pi}| \ge \eta \eps_0^{-1} =:\eta'.
\end{equation}

Let $\eta'$ be of order $p^{-1}$. As before, for each $a$ we let $k_a\in \Z$ be such that $-p/2 < k_a < p/2$ and that $|\frac{s_a}{2\pi} - \frac{k_a}{p}| \le \frac{1}{2p}$. Then by the assumption of Lemma \ref{lemma:norm:str:d}, $k_a =o(p)$, and also by \eqref{eqn:good:s:d}, as long as $(a,b)\in \CG$ we have 
$$k_{a_1}+\dots+k_{a_{d-1}}+k_{-a_1-\dots-a_{d-1}} \in \{-3d,\dots,3d\},$$
where $A$ is an absolute constant.
As this hold for all $(a,b)\in \CG$ (which occupies most of the tuples of $(a_1,\dots, a_{d-1})$), we will show as in the case $d=3$ the following
\begin{claim} Most of $|k_a|$ are at most $O(d)$.
\end{claim}
\begin{proof}
We say that $a$ is good if the number of tuples $(a_1,\dots, a_{d-1}) \in \CG$ such that $a=-\sum_i a_i$ is at least $(1-\eps_0)p^{d-1}$. It is not hard to see that there are $(1-\eps_0)p$ such good indices. Assume that $a_0$ is such that $|k_{a_0}|$ is largest among the good $a$. Without loss of generality we assume that $k_{a_0}$ is positive.  Consider $(a_1,\dots, a_{d-1})\in \CG$ such that $a_0=-\sum_i a_i$. There are $(1-\eps_0)p^{d-2}$ such tuples, and because most of the indices are good, there are $(1-2\eps)p^{d-2}$ tuples where $a_i$ are good. Because $k_{a_0} \in (-\sum_i k_{a_i}) + \{-3d,\dots,3d\}$, and because $k_{a_0}$ is largest, there must be good $a_i$ such that $k_{a_i}$ is negative. Arguing as in the case $d=3$, assume that the set $\CN$ of $a$ for which $k_a$ is negative has size at least $10d\eps_0 n$, then most of the $k_a$ must be around $-k_{a_0}/(d-1)$. We can find  many  $a''$ for which $k_{a''} \approx k_{a_0}$, and therefore a tuple $(a''_1,\dots,a''_{d-1}) \in \CG$ . Then  $k_{-\sum a''_i}\approx(d-1) k_{a_0}$ , which is larger than $k_{a_0}$, a contradiction.
\end{proof}
The rest of the proof of Lemma \ref{lemma:norm:str:d} can be completed as in $d=3$, hence we omit the details.
\end{proof}

\section{The error term: proof of Proposition \ref{prop:dev:directed}}\label{section:dev} 
Recall that we are working with
$$\sum_{j=0}^{p-1} (\frac{n_j}{n} - \frac{1}{p})^2 > \frac{b \log n}{n}.$$
Our proof here is similar to \cite[Proposition 3.4]{H2}, which can be divided into four cases

\begin{enumerate}[(i)]
\item $\CN_1$ of $(n_0,\dots, n_{p-1}) \in \CN$ with 
$$\max_j |n_j/n-1/p| \le \delta/p;$$
\item $\CN_2$ of $(n_0,\dots, n_{p-1}) \in \CN$ with 
$$(b p \log n) / n < |n_0/n-1| \le \delta/p;$$
\item  $\CN_3$ of $(n_0,\dots, n_{p-1}) \in \CN$ with 
$$|n_0/n-1| < (b p \log n) /n;$$
\item $\CN_4$ of the remaining non-equidistributed $p$-tuples.
\end{enumerate}
We then have

\begin{lemma}\label{lemma:dev:d:124} For $p \ll n^{1/3}$ \footnote{In fact the statements here are also true for $p\ll n^{1/2}$.}, the sum over $(n_0,\dots, n_{p-1}) \notin \CN_3$ is bounded by
$$\sum_{(n_0,\dots, n_{p-1}) \in \CN_1 \cup \CN_2 \cup \CN_4} \binom{n}{n_0,\dots, n_{p-1}} \binom{dn}{dn_0,\dots, dn_{p-1}}^{-1}\Big|\{(\Bu_1,\dots, \Bu_n) \in \CU_{d,p}^n: \Bu_1+\dots+ \Bu_n =(dn_0,\dots, dn_{p-1})\}\Big|$$ 
$$= O(1/n^{d-2}).$$
\end{lemma}
The proof of this is identical to that of \cite[Proposition 3.4]{H2}, hence we omit it.


Our new contribution is that the sum from $\CN_3$ is also insignificant for $p\ll n^{1/3}$, for which we state below.

\begin{lemma}\label{lemma:dev:d:3} For $p \ll n^{1/3}$ we have 
$$\sum_{(n_0,\dots, n_{p-1})\in \CN_3} \binom{n}{n_0,\dots, n_{p-1}} \binom{dn}{dn_0,\dots, dn_{p-1}}^{-1}\Big|\{(\Bu_1,\dots, \Bu_n) \in \CU_{d,p}^n: \Bu_1+\dots+ \Bu_n =(dn_0,\dots, dn_{p-1})\}\Big| =o(1).$$
\end{lemma}

\begin{proof}(of Lemma \ref{lemma:dev:d:3}) 
The treatment here is motivated by Case 3 in the proof of \cite[Proposition 3.4]{H2}, although we introduce some minor modifications. We assume that $n_0 =n-m'$ and $n_1+\dots+ n_{p-1} =m'$, where 
$$m' \le b p \log n.$$
 We list $\CU_{d,p}$ as 
$$\CU_{d,p} =\{\Bw_1,\dots, \Bw_{p^{d}-1}\}, \Bw_1=(d,0,\dots,0).$$
Notice that (where $\Bw(j)$ is the $j$-th coordinate of $\Bw$)
$$\Bw_j(1)+\dots + \Bw_j(p-1) \ge 2, 2 \le j \le p^{d-1}.$$
 For short, we let $m$ be the number of non-zero vectors (and $n-m$ be the number of $\Bw_1$) in $\Bu_1,\dots, \Bu_n$, and $n_0' =m- n_2-\dots- n_{p-1}$. We have $2 m \le dm'$, so 
$$m \le dm'/2.$$ 
This just shows that the number of $\Bw_1$ in $(\Bu_1,\dots, \Bu_n)$ must be at least $n-dm'/2$. We thus have $\binom{n}{m}$ ways to arrange the $\Bu_i$ to be $\Bw_1$. After that we have a sum of $\Bu_1+\dots+ \Bu_m = (dn_0',dn_1,\dots, dn_{p-1})$ and $n_0'+n_1+\dots+n_{p-1}=m$ where $n' \le (d-2)m/d$. 

As we are interested in $\binom{n}{n_0,\dots, n_{p-1}} \binom{dn}{dn_0,\dots, dn_{p-1}}^{-1} \times |\{(\Bu_1,\dots, \Bu_m): \sum = (dn_0',dn_1,\dots, dn_{p-1})\}|$, we can rewrite the first factor as 
$$\frac{n!}{n_0! \dots n_{p-1}!} = \frac{n!n_0'!}{m! n_0!}  \frac{m!}{n_0'! \dots n_{p-1}!} =  \frac{n!n_0'!}{m! n_0!}  \binom{m}{n_0',\dots, n_{p-1}}$$
and we can write the second factor as 
$$ \frac{dn_0! \dots dn_{p-1}!}{dn!} = \frac{dm! dn_0!}{dn! dn_0'!}  \frac{dn_0'! dn_1! \dots dn_{p-1}!}{dm!} = \frac{dm! dn_0!}{dn! dn_0'!}   \binom{dm}{dn_0',\dots, dn_{p-1}}^{-1}.$$
Note that $n_0 = n -\sum_{i=1}^{p-1} n_i = n-(m-n_0')$. Hence
$$
 \frac{n!n_0'!}{m! n_0!}  \approx n^{m-n_0'}/m^{m-n_0'}  \approx (n/m)^{m-n_0'} 
$$
and
$$ \frac{dm! dn_0!}{dn! dn_0'!}  \approx (dn)^{-d(m-n_0')} (dm)^{d(m-n_0')} \approx (n/m)^{-d(m-n_0')}.$$
Hence
$$\frac{n!n_0'!}{m! n_0!}  \frac{dm! dn_0!}{dn! dn_0'!} \approx  (m/n)^{(d-1)(m-n_0')}.$$

We next apply the following analog of \cite[Proposition 3.2]{H2} (where $n$ is replaced by $m$)

\begin{lemma}\label{lemma:dev:m} We have 
$$\binom{m}{n_0',\dots, n_{p-1}} \binom{dm}{dn_0',\dots, dn_{p-1}}^{-1} |(\Bu_1,\dots, \Bu_m), \Bu_1+\dots+ \Bu_m = (dn_0',dn_1,\dots, dn_{p-1})| =O(e^{O(p)}).$$
 \end{lemma}
This result is a special case of Lemma \ref{lemma:dev:m'} to be stated below. 

By this result, in total we obtain
$$\sum_{m\le b p \log n} \sum_{\substack{(n_0',\dots, n_{p-1}),\\ n_0'+\dots + n_{p-1} =m}} e^{O(p)} \binom{n}{m}  (m/n)^{(d-1)(m-n_0')}$$
$$=\sum_{m\le b p \log n} \sum_{n_0' \le (d-2)m/d}  \sum_{\substack{n_1,\dots, n_{p-1}),\\ n_1+\dots + n_{p-1} =m-n_0'}} e^{O(p)} \binom{n}{m}  (m/n)^{(d-1)(m-n_0')}$$
$$ \le \sum_{m\le b p \log n}  \sum_{n_0' \le (d-2)m/d}   \binom{p+m-n_0'}{p-1} e^{O(p)} \binom{n}{m}  (m/n)^{(d-1)(m-n_0')}.$$
{\bf Case 1.} We see that the contribution is small for $p/\log n\ll m \le b p \log n$ because $ \binom{p+m-n_0'}{p-1}  \le (e (p+m)/p)^p$ and $\binom{n}{m} \le (en/m)^m$ , while $(m/n)^{(d-1)(m-n_0')} \le (m/n)^{2(d-1)m/d}$.

{\bf Case 2.} For $m\ll p /\log n$, Lemma \ref{lemma:dev:m} is not powerful because as $n_0'+n_1+\dots+n_p=m$, many $n_i$ are zero. To amend this, let $\ell$ be the number of nonzero $n_{i_j}$, then $0\le \ell \le m$ and $n_{i_1}+\dots+n_{i_\ell} =m$. There are $\binom{p}{\ell}$ ways to choose the $i_1,\dots, i_\ell$.  So $\binom{m}{n_0',\dots, n_{p-1}} \binom{dm}{dn_0',\dots, dn_{p-1}}^{-1}$ becomes $\binom{m}{n_{i_1},\dots, n_{i_\ell}} \binom{dm}{dn_{i_1},\dots, dn_{i_\ell}}^{-1}$. 

Also, as $\Bu_1+\dots+ \Bu_m = (dn_0',dn_1,\dots, dn_{p-1})$, the vectors $\Bu_i$ are from the set $\CU_{d; i_1,\dots, i_\ell}$ of vectors $\Phi(x_1,\dots, x_d)$ where $x_i \in \{i_1,\dots, i_\ell\}$ and $\sum_i x_i =0$. Note that this set $\CU_{d; i_1,\dots, i_\ell}$ has at most $\ell^{d-1}$ elements, so 
$$|(\Bu_1,\dots, \Bu_m): \Bu_1+\dots+ \Bu_m = (dn_0',dn_1,\dots, dn_{p-1}|) = |\CU_{d; i_1,\dots, i_\ell}|^m  \times \P(X_1+\dots+X_m= (dn_{i_1},\dots, dn_{i_\ell})),$$ 
where $X_i$ are sampled uniformly from the set $\CU_{d;i_1,\dots, i_\ell}$.

\begin{lemma}\label{lemma:dev:m'} We have 
$$\binom{m}{n_{i_1},\dots, n_{i_\ell}} \binom{dm}{dn_{i_1},\dots, dn_{i_\ell}}^{-1}   |(\Bu_1,\dots, \Bu_m), \Bu_1+\dots+ \Bu_m = (dn_{i_1},\dots, dn_{i_\ell})| \le e^{O(\ell)}.$$
 \end{lemma}
 
 Assuming Lemma \ref{lemma:dev:m'} for a moment, by summing over $(n_{0},\dots, n_{\ell})$ as a partition of $m$ (of which there are at most $\binom{m+\ell-1}{m}$) and over the choices of $i_1,\dots, i_\ell$ we have
$$\sum_{m\ll p /\log n} \sum_{n_0' \le (d-2)m/d} \sum_{1\le \ell \le m} \binom{p}{\ell} \binom{m+\ell-1}{m} e^{O(\ell)} \binom{n}{m}  (m/n)^{(d-1)(m-n_0')}.$$
We remark that for $\ell \le m\le p/\log n$ we have $\binom{p}{\ell} \le (ep/\ell)^\ell \le (e p/m)^m$ and $\binom{m+\ell+1}{m}<2^m$, and $\binom{n}{m} \le (en/m)^m$, while $ (m/n)^{(d-1)(m-n_0')} \le  (m/n)^{2(d-1)m/d}  \le (m/n)^{4m/3}$ (where $d=3$ is the worst case). Hence the sum above is trivially bounded by
\begin{align*}
(p/\log n) \times m \times m \times  (e p/m)^m 2^m (en/m)^m  (m/n)^{4m/3} &\le (p/\log n)m^2 (2e^2)^m \times  (p/n^{1/3})^m (1/m)^{2m/3}=o(1) 
\end{align*}
where we used the crucial fact that $p \ll n^{1/3}$. This complete the proof of Lemma \ref{lemma:dev:d:3}.
\end{proof}

 \begin{proof}(of Lemma \ref{lemma:dev:m'}) Without loss of generality assume $\{i_1,\dots, i_\ell\} = \{0,\dots, \ell-1\}$. Using \eqref{Stirling},
  \begin{align*}
\binom{n}{n_0,\dots, n_{\ell-1}}\frac{\prod_{j=0}^{\ell-1} (d n_j)!}{(dm)!} &=\frac{m!}{\prod_j n_j!} \frac{\prod_{j=0}^{\ell-1} (d n_j)!}{(dm)!}\\
&= e^{\frac{1}{12m}- \frac{1}{12dm} + \sum_j \frac{1}{12dn_j} - \frac{1}{12n_j} } \times (\sqrt{d})^{\ell-1} \times [\prod_{j=0}^{\ell-1} (\frac{n_j}{m})^{n_j}]^{d-1}.
\end{align*} 
Write $\mfn_j = \frac{n_j}{m}$ and $h_j= \mfn_j - 1/\ell$. We then write the expression in Lemma \ref{lemma:dev:m'} as 
 $$ e^{O(\ell)} e^{m I(\mathfrak{n}_{0}, \dots, \mathfrak{n}_{\ell})}$$
 where the $e^{O(\ell)}$ term comes from $ (\sqrt{d})^{\ell-1}$ and  
 $$ I(\mathfrak{n}_0, \dots, \mathfrak{n}_{\ell}) = \log |\CU_{d; 0,\dots,\ell-1}| + (d-1) \sum_j \mathfrak{n}_j \log \mathfrak{n}_j + \inf_{\Bt \in \R^\ell} (\log \E e^{\lang \Bt, X\rang } - d \lang \Bt, \mathfrak{n} \rang).$$
It remains to show the following
\begin{claim}\label{claim:dev:neg} We have $I(\mathfrak{n}_0, \dots, \mathfrak{n}_{\ell}) \le 0$ and equality holds only if $ \mathfrak{n}_j=1/\ell$ for all $j$ or $ \mathfrak{n}_0=1$ and  $\mathfrak{n}_j=0$ for $j \neq 0$.
\end{claim}
To show this claim we follow \cite[Prop 3.3]{H1}. Choose $\Bt= \frac{d-1}{d} (\log \mathfrak{n}_0, \dots, \log \mathfrak{n}_{\ell-1})  +  \frac{\log |\CU|}{d} \1$. Then $I$ is bounded by 
$$ \log |\CU_{d; 0,\dots,\ell-1}| + (d-1) \sum_j \mathfrak{n}_j \log \mathfrak{n}_j + \log \E e^{\lang \Bt, X\rang } - d \lang \Bt, \mathfrak{n} \rang  =  \log \E e^{\lang \Bt, X\rang },$$
where $X$ is sampled uniformly from $\CU_{d; 0,\dots,\ell-1}$. We will show that $\E e^{\lang \Bt, X\rang } \le 1$. Let $\Bw_1,\dots, \Bw_{\ell^{d-1}}$ be an enumeration of $\CU_{d; 0,\dots,\ell-1}$, we have
\begin{align*} \
\E e^{\lang \Bt, X\rang } &= \frac{1}{|\CU|}  \sum_{j\in \CU}  e^{\lang \Bw_j , \Bt \rang } =\frac{1}{|\CU|}  \sum_{j\in \CU}  e^{\sum_{k=0}^{\ell-1} \Bw_j(k) ((d-1)/d) \log \mathfrak{n}_{k} + \Bw_j(k) \log |\CU|/ d}  \\
& = \sum_{j\in \CU}  e^{\sum_{k=0}^{\ell-1} \Bw_j(k) ((d-1)/d) \log \mathfrak{n}_{k}} =  \sum_{j\in \CU}  \prod_{k=0}^{\ell -1} \mathfrak{n}_{k}^{ ((d-1)/d)  \Bw_j(k) }\\
& = \sum_{\substack{\Ba =(a_1,\dots, a_d) \in  \{0,\dots, \ell-1\}^d,\\ \sum_i a_i=0}}  \prod_{k=0}^{\ell -1} \mathfrak{n}_{k}^{ ((d-1)/d)  \Phi(\Ba)(k) }\\
& = \sum_{\substack{\Ba \in  \{0,\dots, \ell-1\}^d,\\ \sum_i a_i=0}}  \prod_{k=0}^{\ell -1} \mathfrak{n}_{k}^{ ((d-1)/d)  \sum_{r=1}^d 1_{a_r=k}}\\
&= \sum_{\substack{\Ba \in  \{0,\dots, \ell-1\}^d,\\ \sum_i a_i=0}}  \prod_{r=1}^{d} {\mathfrak{n}_{a_r}}^{ (d-1)/d }.
\end{align*}
Next, note that $\sum_{i\in \{0,\dots, \ell-1\}}\mathfrak{n}_{i}=1$ and 
$$ \prod_{r=1}^{d} {\mathfrak{n}_{a_r}}^{ (d-1)/d } \le \frac{1}{d}\sum_{r=1}^d \prod_{\substack{1\le s\le d,\\ s\neq r}}\mathfrak{n}_{a_s}.$$
So 
\begin{align*} \
\E e^{\lang \Bt, X\rang } &= \sum_{\substack{\Ba  \in  \{0,\dots, \ell-1\}^d,\\ \sum_i a_i=0}}  \prod_{r=1}^{d} {\mathfrak{n}_{a_r}}^{ (d-1)/d }\\
&\le \frac{1}{d} \sum_{\substack{\Ba  \in  \{0,\dots, \ell-1\}^d,\\ \sum_i a_i=0}}  \sum_{r=1}^d \prod_{\substack{1\le s\le d,\\ s\neq r}}\mathfrak{n}_{a_s}\\
& \le (\sum_{i\in \{0,\dots, \ell-1\}}\mathfrak{n}_{i})^d =1.
\end{align*}

\end{proof}

{\bf Acknowledgements.} The authors are thankful to V. Lev and G. Petridis for helpful discussion and to Eric Lu for programming assistance.



\newcommand{\etalchar}[1]{$^{#1}$}

\end{document}